\documentclass[11pt]{article}
\usepackage{url}
\usepackage{amsmath}
\usepackage{array}
\usepackage{amsfonts}
\usepackage{enumerate}
\usepackage{graphicx}
\usepackage{wrapfig} 
\usepackage{xcolor} 
\usepackage{stmaryrd}
\usepackage{tipa}
\usepackage{float}
\usepackage{bbm} 
\usepackage{tikz}
\usepackage{algpseudocode}
\usepackage{algorithm}
\usepackage{mathtools}
\usepackage{afterpage}
\usepackage{pdflscape}
\usepackage{capt-of}
\usepackage{epstopdf}
\usepackage{setspace}
\usepackage{changepage}
\usepackage{tabularx}
\usepackage{geometry}
\usepackage{fancyhdr}
\usepackage{amssymb}
\usepackage{tikz-cd}

\newcommand{\mz}{\mathbb{Z}}
\newcommand{\mr}{\mathbb{R}}
\newcommand{\mn}{\mathbb{N}}

\newcommand{\mc}{\mathbb{C}}
\newcommand{\del}{\partial}
\newcommand{\twid}{{\raise.17ex\hbox{$\scriptstyle\sim$}}}
\DeclareMathOperator{\im}{Im}

\DeclareMathOperator{\hfs}{hfs}

\DeclareMathOperator{\mvr}{V_\mr}

\usepackage{latexsym}
\usepackage{amsmath}
\usepackage{amsthm}
\usepackage{amssymb,amsfonts,amscd}
\usepackage{graphicx}
\usepackage{epstopdf}
\usepackage{multirow}
\usepackage{mathabx} 
\usepackage[colorlinks=true]{hyperref}
\usepackage{color}
\usepackage{subcaption}
\usepackage{cite}
\usepackage{array} 
\usepackage{booktabs}
\usepackage{chngcntr}
\usepackage[noabbrev,capitalize]{cleveref} 

\newcommand*{\email}[1]{\href{mailto:#1}{\nolinkurl{#1}}}

\newcommand{\sL}{{\mathcal L}}

\newcommand\bC{{\mathbb C}}

\newcommand\bP{{\mathbb P}}

\newcommand\bR{{\mathbb R}}

\DeclareMathOperator*{\rank}{\rm rank}

\newtheorem{theorem}{Theorem}[section]
\newtheorem{corollary}[theorem]{Corollary}
\newtheorem{proposition}[theorem]{Proposition}
\newtheorem{lemma}[theorem]{Lemma}

\theoremstyle{definition}
\newtheorem{definition}[theorem]{Definition}

\theoremstyle{remark}
\newtheorem{rem}[theorem]{Remark}
\newtheorem{example}[theorem]{Example}

\newcommand{\done}{\hfill $\triangleleft$}

\usetikzlibrary{arrows}
\newcommand{\mypoints}{1.634829e-01/1.763406e-01,-2.344660e-01/2.051455e-01,-9.479990e-01/2.161792e-01,
-3.884802e-01/3.037904e-01,2.703894e-01/3.047601e-01,-8.758754e-01/3.085825e-01,
-4.978205e-01/3.527786e-01,-5.883462e-01/3.774844e-01,-7.147397e-01/3.817408e-01,
3.854243e-01/4.536600e-01,-6.091223e-01/-3.808246e-01,-7.665967e-01/-3.703569e-01,
-4.978205e-01/-3.527786e-01,-8.622812e-01/-3.199967e-01,2.703894e-01/-3.047601e-01,
-3.884802e-01/-3.037904e-01,-9.293902e-01/-2.469623e-01,-2.344660e-01/-2.051455e-01,
1.634829e-01/-1.763406e-01,-1.505363e-01/-1.387438e-01,-9.883641e-01/-1.066145e-01,
8.635207e-02/-9.000323e-02,0.0/0.0,-1.0e00/0.0,
-9.883641e-01/1.066145e-01,-1.505363e-01/1.387438e-01,-8.007185e-01/3.574484e-01
}
\newcommand{\cechcomplex}[1]{%
\begin{tikzpicture}[line cap=round,line join=round,>=triangle
  45,x=1cm,y=1cm,scale=2.0]
\foreach \x/\y in \mypoints \fill [color=cyan!60] (\x,\y) circle (#1cm);
\foreach \xa/\ya in \mypoints
\foreach \xb/\yb in \mypoints
{
  \pgfmathtruncatemacro{\z}{(\xa-\xb)^2+(\ya-\yb)^2<(#1)^2*4}
  \ifnum \z = 1
  \draw[color=black!80] (\xa,\ya) -- (\xb,\yb);
  \fi
}
\foreach \x/\y in \mypoints \fill [color=blue] (\x,\y) circle (0.04cm);
\end{tikzpicture}%
}

\newcommand{\cechcomplextwo}[1]{%
\begin{tikzpicture}[line cap=round,line join=round,>=triangle
  45,x=1cm,y=1cm,scale=2.0]
\foreach \xa/\ya in \mypoints
\foreach \xb/\yb in \mypoints
\foreach \xc/\yc in \mypoints
{
  \pgfmathtruncatemacro{\z}{(\xa-\xb)^2+(\ya-\yb)^2<(#1)^2*4}
  \pgfmathtruncatemacro{\w}{(\xa-\xc)^2+(\ya-\yc)^2<(#1)^2*4}
  \pgfmathtruncatemacro{\v}{(\xb-\xc)^2+(\yb-\yc)^2<(#1)^2*4}
  \ifnum \z = 1
  \ifnum \w = 1
  \ifnum \v = 1
  \draw[fill=cyan!25] (\xa,\ya) -- (\xb,\yb) -- (\xc,\yc) -- cycle;
  \fi
  \fi
  \fi
}
\foreach \xa/\ya in \mypoints
\foreach \xb/\yb in \mypoints
{
  \pgfmathtruncatemacro{\z}{(\xa-\xb)^2+(\ya-\yb)^2<(#1)^2*4}
  \ifnum \z = 1
  \draw[color=black!80] (\xa,\ya) -- (\xb,\yb);
  \fi
}
\foreach \x/\y in \mypoints \fill [color=red!80] (\x,\y) circle (0.04cm);
\end{tikzpicture}%
}

\newcommand{\onesimplexpoints}{0.0/0.0,1.0/0.0}
\newcommand{\onesimplex}[0]{%
\begin{tikzpicture}[line cap=round,line join=round,>=triangle
  45,x=1cm,y=1cm,scale=2.0]
\foreach \xa/\ya in \onesimplexpoints
\foreach \xb/\yb in \onesimplexpoints
{
  \pgfmathtruncatemacro{\z}{(\xa-\xb)^2+(\ya-\yb)^2>0}
  \ifnum \z = 1
  \draw[color=black!80] (\xa,\ya) -- (\xb,\yb);
  \fi
}

\foreach \x/\y in \onesimplexpoints \fill [color=red!80] (\x,\y) circle (0.04cm);
\end{tikzpicture}%
}

\newcommand{\twosimplexpoints}{0.0/0.0,1.0/0.0,0.5/1.0}
\newcommand{\twosimplex}[0]{%
\begin{tikzpicture}[line cap=round,line join=round,>=triangle
  45,x=1cm,y=1cm,scale=2.0]
\foreach \xa/\ya in \twosimplexpoints
\foreach \xb/\yb in \twosimplexpoints
{
  \pgfmathtruncatemacro{\z}{(\xa-\xb)^2+(\ya-\yb)^2>0}
  \ifnum \z = 1
  \draw[color=black!80] (\xa,\ya) -- (\xb,\yb);
  \fi
}

\foreach \xa/\ya in \twosimplexpoints
\foreach \xb/\yb in \twosimplexpoints
\foreach \xc/\yc in \twosimplexpoints
{
  \pgfmathtruncatemacro{\z}{(\xa-\xb)^2+(\ya-\yb)^2>0}
  \pgfmathtruncatemacro{\w}{(\xa-\xc)^2+(\ya-\yc)^2>0}
  \pgfmathtruncatemacro{\v}{(\xb-\xc)^2+(\yb-\yc)^2>0}
  \ifnum \z = 1
  \ifnum \w = 1
  \ifnum \v = 1
  \draw[fill=cyan!25] (\xa,\ya) -- (\xb,\yb) -- (\xc,\yc) -- cycle;
  \fi
  \fi
  \fi
}
\foreach \x/\y in \twosimplexpoints \fill [color=red] (\x,\y) circle (0.04cm);
\end{tikzpicture}%
} 

\newcommand{\threesimplexpoints}{0.0/0.0/0.0,1.0/0.0/0.0,0.5/1.0/0.0,0.5/0.5/1}
\newcommand{\threesimplex}[0]{%
\begin{tikzpicture}[line cap=round,line join=round,>=triangle
  45,x=1cm,y=1cm,scale=2.0]
\foreach \xa/\ya/\za in \threesimplexpoints
\foreach \xb/\yb/\zb in \threesimplexpoints
{
  \pgfmathtruncatemacro{\z}{(\xa-\xb)^2+(\ya-\yb)^2+(\za-\zb)^2>0}
  \ifnum \z = 1
  \draw[color=black!80] (\xa,\ya) -- (\xb,\yb);
  \fi
}

\foreach \xa/\ya/\za in \threesimplexpoints
\foreach \xb/\yb/\zb in \threesimplexpoints
\foreach \xc/\yc/\zc in \threesimplexpoints
{
  \pgfmathtruncatemacro{\z}{(\xa-\xb)^2+(\ya-\yb)^2+(\za-\zb)^2>0}
  \pgfmathtruncatemacro{\w}{(\xa-\xc)^2+(\ya-\yc)^2+(\za-\zc)^2>0}
  \pgfmathtruncatemacro{\v}{(\xb-\xc)^2+(\yb-\yc)^2+(\zb-\zc)^2>0}
  \ifnum \z = 1
  \ifnum \w = 1
  \ifnum \v = 1
  \draw[fill=cyan!25] (\xa,\ya,\za) -- (\xb,\yb,\zb) -- (\xc,\yc,\zc) -- cycle;
  \fi
  \fi
  \fi
}
\foreach \x/\y/\z in \threesimplexpoints \fill [color=red] (\x,\y,\z) circle (0.04cm);
\end{tikzpicture}%
}

\begin{document}
	\title{Sampling real algebraic varieties for topological data analysis}
	
	\author{
		Emilie Dufresne\thanks{School of Mathematical Sciences, University of Nottingham, University Park, Nottingham, NG7 2RD (\email{emilie.dufresne@nottingham.ac.uk}).} \and
		Parker B. Edwards\thanks{Department of Mathematics, PO Box 118105, University of Florida, Gainesville, FL 32611 (\email{pedwards@ufl.edu})} \and
		Heather A. Harrington\thanks{Mathematical Institute, Radcliffe Observatory Quarter, Woodstock Road, Oxford OX2 6GG, United Kingdom (\email{harrington@maths.ox.ac.uk})}  \and
		Jonathan D. Hauenstein\thanks{Department of Applied and Computational Mathematics and Statistics,
			University of Notre Dame, Notre Dame, IN 46556 (\email{hauenstein@nd.edu}, \url{www.nd.edu/\~jhauenst}).}
	}
	
	\maketitle


	\begin{abstract}
		\noindent 
		Topological data analysis (TDA) provides a growing body of tools for computing geometric and topological information about spaces from a finite sample of points.  We present a new adaptive algorithm for finding provably dense samples of points on real algebraic varieties given a set of defining polynomials. The algorithm utilizes methods from numerical algebraic geometry to give formal guarantees about the density of the sampling and it also employs geometric heuristics to reduce the size of the sample. As TDA methods consume significant computational resources that scale poorly in the number of sample points, our sampling minimization makes applying TDA methods more feasible. We provide a software package that implements the algorithm and also demonstrate the implementation with several examples. 
	\end{abstract}
	

	\section{Introduction}\label{Sec:Intro}

	Understanding the geometry and topology of real algebraic varieties is a ubiquitous and challenging problem in applications modelled by polynomial systems. For kinematics problems, geometric insight about configuration spaces can lead to physical insight about the system being modelled (e.g., \cite{cyclotop}), while the geometry of varieties can encode information about the dynamics of biochemical systems (e.g., \cite{Gross2016}). In this paper, we present a new algorithm fulfilling a key step in applying topological data analysis methods (TDA), particularly persistent homology \cite{zcompute}, to real algebraic varieties. We aim to provide and demonstrate a computationally feasible pipeline for applying TDA to real algebraic varieties while maintaining theoretical guarantees.
	
	The most closely related problem to computing persistent homology (PH) of real algebraic varieties is the computation of its Betti numbers. There are two main approaches to this problem: symbolic methods which process polynomial equations directly, and surface reconstruction methods which estimate Betti numbers by constructing spaces from point~samples. 
	
	The complexity of symbolically computing Betti numbers of (complex) projective varieties has been studied \cite{scheiblechner2007complexity} as well as numerically stable homology computations of real projective varieties~\cite{cucker2016computing} and real semialgebraic sets \cite{Basu2006Betti}. 
	Using random sampling of manifolds, 
	an algorithm is provided in \cite{Niyogi2008} for ``learning''
	the homology with complexity bounds in terms of a condition number relating to the curvature and closeness to self-intersections.
	Alternatively, one can intersect a projective variety with a general linear space and obtain points. For a large enough set of ``general'' points obtained this way, \cite{mustactua1998graded} studies the Betti diagram, which provides information about the projective variety. Although an algorithm for obtaining such general points is not given in that paper, the recent paper \cite{RealSample} provides a uniform sampling method for algebraic manifolds using slicing. For a given set of general points, recent work has ``learned" the equations defining the algebraic variety and confirmed the expected homology via persistent homology~computations \cite{breiding2018learning}.
	
	Extensive effort has also produced a large number of surface reconstruction algorithms, particularly for nonsingular surfaces embedded in $\mr^3$. The survey \cite{SurfaceReconstructionSurvey} and articles \cite{AmentaSampling, DeyReconstruction, CycloReconstruction, IncrementalReconstruction} provide a representative, though by no means exhaustive, list of examples. The general format of surface reconstruction algorithms is to take as input a point cloud sampled from an underlying surface or space, and output a simplicial complex (or richer structure) which geometrically estimates the underlying space. Betti numbers computed from the simplicial complex serve as estimates of the numbers for the underlying space. The inherent limitations on the underlying space required by these methods exclude their use as general tools for analyzing real algebraic varieties. Indeed, all of the methods listed above either provide no theoretical guarantees on the correctness of the reconstruction, or require that the underlying space is some combination of nonsingular, embedded in $\mr^3$, and/or intrinsically 2-dimensional. 
	
	We focus on using the persistent homology pipeline to analyze real varieties. Like surface reconstruction methods, the pipeline accepts as input a finite set of points sampled from a space and outputs estimated homological information about the space. More precisely, suppose that we are given defining polynomials for a pure dimensional algebraic set $V\subseteq \mc^N$. The compact set~$V_\mr$ resulting from intersecting $V$
	with a hypercube in $\mr^N$ is the one we wish to sample and analyze. PH captures richer homological information than just Betti numbers, and the theoretical guarantees of PH apply to potentially singular varieties embedded in any dimension. As a trade off, the computational resources required to compute PH quickly become large when more points are added to the sample (e.g., \cite{roadmap}). Both the theoretical framework for the PH pipeline and its computational costs drive the requirements of a suitable sampling algorithm. Among existing sampling approaches, subdivision and reduction sampling methods \cite{SherbrookeSubdivision,MourrainSubdivision} are the most obvious candidates. In their most general format, these methods can take the polynomials defining a real semialgebraic set as input and output a dense sample of points. For PH computations, they exhibit two drawbacks: 
	
	\begin{enumerate}
		\item Sample points in the output need not be especially close to the underlying variety. Input samples with points close to the underlying variety significantly improve the accuracy of PH results.
		\item Adjusting current implementations to reduce the number of sample points in the output is not straightforward. Computational resource requirements for PH scale up quickly with more sample points.
	\end{enumerate}
	
	Our alternative approach for sampling varieties is based on numerical algebraic geometry, with the books \cite{bertinibook,somnag}
	providing a general overview. The algorithm addresses the first point above by constructing provably dense samples with points very close to the underlying variety. The theoretical version of the algorithm can be readily adjusted to incorporate geometric heuristics which significantly reduce the number of points in the final output thereby addressing the second point.
	
	The remainder of the paper is organized as follows: In \Cref{Sec:TDA}, we recall the TDA theory and computational considerations which determine our sampling algorithm's requirements. In \Cref{Sec:Sampling}, we explain the tools from numerical algebraic geometry used in our approach. \Cref{Sec:Generate} details the sampling algorithm, proves its correctness, and discusses the geometric heuristics for sample minimization. Finally, we present examples in \Cref{Sec:Examples} to illustrate how our sampling algorithm can be used in conjunction with TDA to calculate topological information 
	for several~varieties.

	
	\section{Topological data analysis}\label{Sec:TDA}
	
	Topological data analysis is a very active field of research broadly encompassing theory and algorithms which adapt the theoretical tools of topology and geometry to analyze the ``shape'' of data. We concentrate on applying the ``persistent homology pipeline'' popularized by Carlsson in \cite{carltop} and summarized by Ghrist in \cite{ghristtop}. Broader overviews of other TDA methods can be found in the articles \cite{ChazalSurvey,roadmap} and textbooks \cite{HarerBook,OudotBook}. The PH pipeline follows these steps: 
	\begin{enumerate}
		\item Input data is expected to be in the form of a \emph{point cloud} consisting of finitely many points in $\mr^n$ together with their pairwise distances. 
		\item A collection of shapes, simplicial complexes, are constructed out of the input data. The complexes encode the shape of the data at different distance scales. 
		\item Algebraic topological features of the simplicial complexes produced in Step 2 are calculated, compared, and ultimately assembled into a single output summary using the algebraic theory of \emph{persistent homology}. 
	\end{enumerate}
	
	In this section, we briefly recall simplicial complexes and the basics of homology (the textbook~\cite{hatcher} provides detailed information on homology theory). We then summarize the theoretical and computational elements in each step of the PH pipeline that pertain to applying the pipeline to real algebraic varieties.
	
	\subsection{Homology groups}
	
	Algebraic topology studies methods for assigning algebraic structures to topological spaces in such a way that the algebraic structures encode topological information about the space. For a natural number $p\geq 0$, the $p$-th homology group of a space captures information about the number of $p$-dimensional holes in the space. We initially restrict focus to spaces called simplicial complexes which are more amenable to computation. 
	
	\begin{definition} An (abstract) \emph{simplicial complex} is a finite set $\Omega$ of non-empty subsets of~$\mn$ such that $\omega\in \Omega$ implies that every subset of $\omega$ is an element in $\Omega$. If $\Omega$ is an abstract simplicial complex, the elements of the set $\Omega_0 = \cup_{\omega\in \Omega} \omega$ are called the \emph{vertices} of $\Omega$. The \emph{dimension} of $\Omega$ is one less than the size of the largest set in $\Omega$. For simplicial complexes $\Omega$ and $\Omega'$, a simplicial map from $\Omega$ to $\Omega'$ is a map $f:\Omega_0 \to \Omega'_0$ where $\omega\in \Omega$ implies $f(\omega)\in \Omega'$. 
	\end{definition}
	
	This purely combinatorial definition corresponds geometrically to forming spaces by gluing together points, lines, triangles, tetrahedra, and higher dimensional equivalents. Note that a simplicial complex $\Omega$ defines a subspace $|\Omega|$ of some Euclidean space ($|\Omega|$ is a \emph{geometric realization} of~$\Omega$). The vertices $\Omega_0$ correspond to geometric vertices in $|\Omega|$, the 2-element subsets in $\Omega$ to lines, the $3$-element subsets to triangles, etc., as shown in \cref{simplexfig}. The homology groups (with $\mz/2$ coefficients) for the complex $\Omega$ are built in 3 steps.

	\begin{figure}[h]
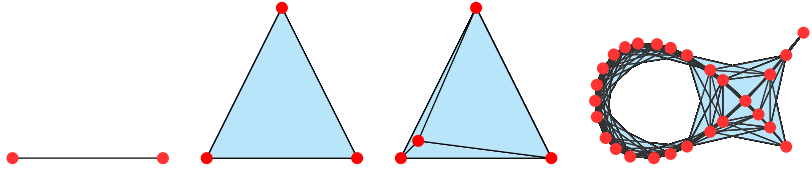

		\centering
		\begin{tabular}{c c c c}
			\onesimplex &
			\twosimplex & 
			\threesimplex & 
			\cechcomplextwo{0.30}
		\end{tabular}
		\caption{\label{simplexfig}
			From left to right: Geometric realizations of a 1-simplex, 2-simplex, 3-simplex (the interior of the tetrahedron is included), and a general simplicial complex.
		}
	\end{figure}
	
	First, encode the complex $\Omega$ algebraically. 
	
	\begin{definition}
		Take $\Omega_p \subseteq \Omega$ for an integer $p\geq 0$ to contain the sets in $\Omega$ with size $p+1$. Members of $\Omega_p$ are \emph{p-simplices} of $\Omega$. The $p$-th chain group of $\Omega$, denoted $C_p(\Omega)$, is the group of formal sums $\sum_{k=1}^n b_k \omega_k$ where $\omega_k\in \Omega_p$ and $b_k\in\mz/2$ for all $k$. Set $C_{-1} = 0$. 
	\end{definition} 
	
	Next, define a simplicial analog to the geometric operation of taking the boundary of a~space. 
	
	\begin{definition} 
		The $p$-th boundary operator is a homomorphism $\del_p: C_{p}(\Omega)\to C_{p-1}(\Omega)$ for each $p\geq 0$. Define $\del_0$ to be the zero map $\del_0: C_0(\Omega) \to 0$. For $p > 0$ and any basis element $\omega \in C_p(\Omega)$, define $\del_p(\omega) = \sum_{ \{\omega' \in \Omega_{p-1}\mid \omega' \subseteq \omega \}} \omega'$. Extending the function linearly from its action on the basis elements of $C_p(\Omega)$ defines $\del_p$ on the entirety of $C_p(\Omega)$. Elements in the kernel of $\del_p$ are called \emph{cycles} and the kernel is denoted $\ker(\del_p) = Z_p(\Omega)$. Elements in the image of $\del_{p+1}$ are called \emph{boundaries} and this group is denoted $\im(\del_{p+1}) = B_p(\Omega)$. 
	\end{definition} 
	
	Finally, capture algebraically, for all dimensions, the geometric intuition that a 1-dimen\-sional loop has no boundary, and encloses a void only if it is not the boundary of a 2-dimensional region. 
	
	\begin{definition} 
		It can be shown that $\del_{p} \circ \del_{p+1} = 0$ for all $p \geq 0$, so that $B_p \subseteq Z_p$. The \emph{$p$-th homology group} of $\Omega$, $H_p(\Omega)$, is the quotient group $Z_p(\Omega)/B_p(\Omega)$. $H_p(\Omega)$ is a finite dimensional vector space, and its rank is called the \emph{$p$-th Betti number} of $\Omega$, $\beta_p(\Omega)$.
	\end{definition}
	
	The elements of the homology groups informally represent loops and higher dimensional equivalents in a space, with $\beta_p(\Omega)$ counting the number of $p$-dimensional holes. A basis element of $H_0(\Omega)$ for a complex $\Omega$ represents a single connected component of $|\Omega|$, a basis element of $H_1(\Omega)$ represents a set of loops which can all be deformed within the space $|\Omega|$ into a loop which encloses the same 2-dimensional void, and basis elements of $H_2(\Omega)$ account for 3-dimensional voids. 
	
	Homology groups behave nicely with respect to simplicial maps, a property called \emph{functoriality}. Consider simplicial complexes $\Omega$ and $\Omega'$. Functoriality implies that for any simplicial map $f:\Omega\to \Omega'$ and $p\geq 0$ there exists an $\mz/2$-linear map $H_p(f):H_p(\Omega)\to H_p(\Omega')$. If $\Omega''$ is a third simplicial complex and $g:\Omega'\to \Omega''$ another simplicial map, then $$H_p(g\circ f) = H_p(g)\circ H_p(f).$$
	
	We can construct homology groups from topological spaces directly without using simplicial complexes. This more general \emph{singular homology} construction ($H_p^\text{sing}$) implies the existence of a $\mz/2$-linear map $H_p^{\text{sing}}(X) \to H_p^{\text{sing}}(Y)$ induced by any continuous function $f:X\to Y$. $H_p^{\text{sing}}$~also retains nice behavior with respect to composition of continuous functions. Standard results in algebraic topology show that these two different notions of homology agree where they are both defined. We will not distinguish between singular and simplicial homology subsequently.
	
	\subsection{Building simplicial complexes from data}
	
	Let \mbox{$\mathcal{X} \subseteq \mr^n$} 
	be a finite point cloud serving as input to the PH pipeline. A natural way to estimate shape information from $\mathcal{X}$ is to build a simplicial complex from $\mathcal{X}$ based on the distances between its points.
	
	\begin{definition}\label{complexes} Let $\mathcal{X}$ be a finite subset of a metric space $Y$ and $\epsilon$ be a real number. The \emph{\u{C}ech complex} for $\mathcal{X}$ with parameter $\epsilon$, $C_\epsilon(\mathcal{X})$, is a simplicial complex such that: 
		\begin{itemize}
			\item If $\epsilon < 0$, $C_\epsilon(\mathcal{X})$ is defined directly to be $\emptyset$.
			\item The vertex set of $C_\epsilon(\mathcal{X})$ is $\mathcal{X}$. 
			\item A set $x\subseteq \mathcal{X}$ belongs to $C_\epsilon(\mathcal{X})$ if there exists a point $y\in Y$ such that distance between $y$ and any point in $X$ is at most $\epsilon$. 
		\end{itemize}
		
		The \emph{Vietoris-Rips complex} for $\mathcal{X}$ with parameter $\epsilon$, denoted $R_\epsilon(\mathcal{X})$, is a simplicial complex that fulfills an alternative version of condition 3 above:
		
		\begin{enumerate}
			\item[*] A set $x\subseteq \mathcal{X}$ belongs to $R_\epsilon(\mathcal{X})$ if the distance between any two points in $x$ is at most $\epsilon$.
		\end{enumerate}

	\end{definition} 
	
	The \u{C}ech complex is closely connected with the 
	geometric operation of ``thickening" a finite point 
	cloud $\mathcal{X} \subseteq\mr^n$. Given a parameter $\epsilon\geq 0$ we 
	can replace the space $\mathcal{X}$ with the union 
	of closed balls of radius $\epsilon$ in $\mr^n$ 
	centered at points in $\mathcal{X}$. The Nerve Theorem (see e.g. \S{4G.3} 
	\cite{hatcher}) guarantees that the singular homology 
	of this version of $\mathcal{X}$ which has been thickened by 
	$\epsilon$ is isomorphic to the simplicial homology 
	of $C_\epsilon(\mathcal{X})$. Calculating \u{C}ech complexes for reasonably sized point clouds presents computational issues, as a large number of intersections of balls around points must be checked. Vietoris-Rips complexes estimate \u{C}ech complexes, and can be constructed more easily since only distances between pairs of points must be checked. See \S{3.2} of \cite{HarerBook} for more computational details about constructing these complexes. The following interleaving result precisely describes the manner in which the Vietoris-Rips complexes estimate \u{C}ech complexes.
	
	\begin{theorem}[de Silva and Ghrist \cite{coveragetop}]\label{interthm}
		If $\mathcal{X}$ is a finite set of points in $\mr^n$ and $\epsilon > 0$ there is a chain of inclusions
		$$ C_{\frac{\epsilon'}{2}}(\mathcal{X}) \subseteq R_{\epsilon '}(\mathcal{X}) \subseteq C_\epsilon(\mathcal{X}) \subseteq R_{2\epsilon}(\mathcal{X})$$
		whenever $\frac{\epsilon}{\epsilon'} \geq \frac{1}{2}\sqrt{\frac{2n}{n+1}}$.
	\end{theorem}

	\subsection{Persistent homology}
	
	Given a point cloud $\mathcal{X}$ sampled evenly from nearby some underlying space $X\subseteq\mr^n$ as input, we could ask which parameter $\epsilon$ produces homology $H_p(C_\epsilon(\mathcal{X}))$ that most closely matches the homology~$H_p(X)$ of the underlying space. TDA methods sidestep this question, and instead consider all of the homology groups $H_p(C_\epsilon(X))$ simultaneously. Persistent homology provides an algebraic framework for tracking homology features as the parameter value changes. We summarize the categorical approach to persistent homology introduced in \cite{catphom}. The central objects of study in this framework are called \emph{persistence modules}. 
	
	\begin{definition} Let $k$ be a field. A \emph{persistence module} is a functor $F:(\mr,\leq)\to{\textbf{Vect}_k}$ from the poset $(\mr,\leq)$ to the category $\textbf{Vect}_k$ consisting of vector spaces over $k$ with linear maps between them. Explicitly, $F$ is determined by: 
		\begin{itemize}
			\item A k-vector space $F(\epsilon)$ for every $\epsilon\in\mr$
			\item A linear map $F(\epsilon\leq\epsilon'):F(\epsilon)\to F(\epsilon')$ for every pair of real numbers $\epsilon \leq \epsilon'$ such that: 
			\begin{itemize}
				\item $F(\epsilon\leq\epsilon)$ is the identity map from $F(\epsilon)$ to itself 
				\item Given real numbers $\epsilon \leq \epsilon' \leq \epsilon''$, $F(\epsilon \leq \epsilon'') = F(\epsilon'\leq\epsilon'')\circ F(\epsilon\leq\epsilon')$
			\end{itemize}
		\end{itemize}
		
		\noindent If $F$ and $G$ are both persistence modules, their direct sum $F\oplus G$ is a persistence module where $(F\oplus G)(\epsilon) = F(\epsilon)\oplus G(\epsilon)$ and similarly $(F\oplus G)(\epsilon\leq\epsilon') = F(\epsilon\leq\epsilon')\oplus G(\epsilon\leq\epsilon')$. 
		
		\noindent $F$ is (naturally) isomorphic to $G$, $F\cong G$, if for all real numbers $\epsilon \leq \epsilon'$ there exist isomorphisms from $F(\epsilon)\to G(\epsilon)$ and $F(\epsilon')\to G(\epsilon')$ such that the following diagram commutes: 
		\begin{center}
			\begin{tikzcd} 
				F(\epsilon) \arrow{d} \arrow{r}{F(\epsilon\leq\epsilon')} &
				F(\epsilon') \arrow{d} \\
				G(\epsilon) \arrow{r}{G(\epsilon\leq\epsilon')} &
				G(\epsilon')
			\end{tikzcd} 
		\end{center} 
	\end{definition}

	\begin{definition}
		\noindent A point $\epsilon\in\mr$ is \emph{regular} for a persistence module $F$ if there exists an interval $I\subseteq\mr$ where $\epsilon\in I$ and $F(a\leq b)$ is an isomorphism for all pairs $a\leq b \in I$. Otherwise $\epsilon$ is \emph{critical}. A functor is \emph{tame} if it has finitely many critical values. 
	\end{definition} 
	
	\begin{example}
		For any finite point cloud $\mathcal{X} \subseteq\mr^n$ and real numbers $0\leq \epsilon \leq \epsilon'$, it follows directly from the definition that $C_\epsilon(\mathcal{X}) \subseteq C_{\epsilon'}(\mathcal{X})$. Regarding the subset inclusion as an inclusion map and fixing $p\geq 0$ results in a sequence of vector spaces and $\frac{\mz}{2}$-linear maps $H_p(C_\epsilon(\mathcal{X})) \hookrightarrow H_p(C_{\epsilon'}(\mathcal{X}))$ from the functoriality of $H_p$. The assignment $\epsilon\mapsto H_p(C_\epsilon(\mathcal{X}))$ along with these linear maps induced by inclusion defines a tame persistence module $H_pC_\bullet(\mathcal{X})$, which we will denote by $HC$. An analagous persistence module exists for the Vietoris-Rips complex, which we will denote by $VR$. \done
	\end{example} 
	
	\begin{example}
		Let $I$ be an interval of the form $[a,b),(a,b],$ or $(a,b)$ where $a,b\in\bar{\mr} = \mr\cup\{-\infty,\infty\}$ and let $k$ be a field. The persistence module $\chi_I$ maps $\epsilon \in \mr$ to the vector space $k$ if $\epsilon \in I$, and maps~$\epsilon$ to the trivial vector space $0$ otherwise. For any real numbers $\epsilon \leq \epsilon'$, define $\chi_I(\epsilon\leq\epsilon')$ to be the identity map if both $\epsilon,\epsilon'\in I$, and the trivial map otherwise. \done
	\end{example}

	\begin{figure}
		\centering
		\includegraphics[scale=0.4]{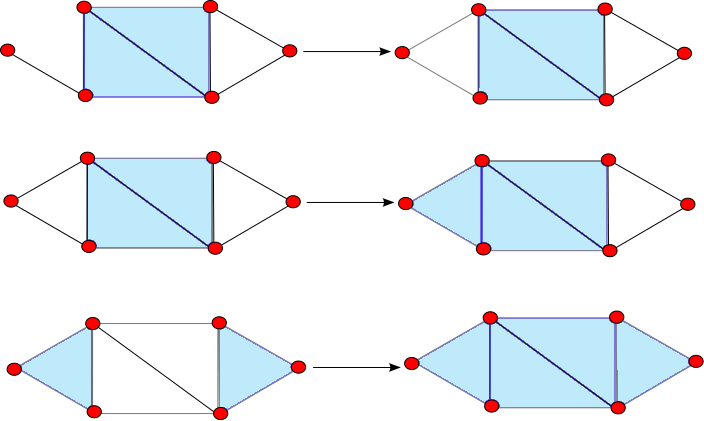}
		\caption{\label{eventsfig}
			Different events in the life of a homology feature. In the top row a feature is born, in the middle two features merge, and in the bottom a feature dies.
		}
	\end{figure}

	\begin{theorem}[Fundamental Theorem of Persistent Homology]\label{fundthm} Let $J$ be a tame persistence module. Then there is a finite set\footnote{More precisely, this set is a \emph{multiset}, which is a set where individual elements can occur more than once.} $\mathcal{B}_J$ of intervals on the real line, $\mathcal{B}_J = \{ I_1,\dots,I_l\}$, such~that
		$$ J \cong \chi_{I_1} \oplus \dots \oplus \chi_{I_l}$$
		and this decomposition is unique up to reordering of the intervals. 
	\end{theorem}
	
	Let $J$ be any tame persistence module and $B_J$ its corresponding set of intervals as in the Fundamental Theorem. The algebraic features encoded in the persistence module $J$ can be visualized (see \cref{barfig}) as a \emph{barcode} or a \emph{persistence diagram}. The set $\mathcal{B}_J$ is called the \emph{barcode} associated to $J$. The \emph{persistence diagram} of $J$, denoted $DJ$, is the set of points $(a,b)\in\bar{\mr}^2$ where $a$ is the left endpoint of an interval $I\in\mathcal{B}_J$, $b$ is the right endpoint of $I$ and $a,b\in\bar{\mr} = \mr\cup\{-\infty,\infty\}$. Note that the points $DJ$ which have a large straight-line distance to the diagonal correspond to long bars in the barcode. $DJ$ also contains all points of the form $(c,c)$ for all $c\in\bar\mr$.
	
	The original algebraic version of \cref{fundthm} for persistent homology appears in \cite{zcompute}, and a categorical version in \cite{catphom}. Each interval in the barcode of a tame functor can be viewed as describing the range of parameter values through which a single independent feature in the module persists. For a module like $H_pC_\bullet(\mathcal{X}) = HC$, an interval $[a,b)$ in the barcode corresponds to a p-dimensional void that first appears at parameter value $a$ and is ``filled in" by $(p+1)$-dimensional simplices at parameter value $b$. See \cref{eventsfig}. 
	
	Persistence diagrams for modules arising from the homology of finite simplicial complexes can be computed via the Persistence Algorithm (see e.g. \cite{harertop} \s VII.1). Note that intervals in the barcodes for such modules always have the form $[a,b)$ for some $a\leq b \in \bar\mr$. Persistence diagrams therefore contain the same amount of information as barcodes for these modules. In the worst case, the computational complexity for computing the persistence of $H_pR_\bullet(\mathcal{X}) = VR$ scales with the maximum number of $p+1$ simplices in $R_\epsilon$ attained at any parameter value~$\epsilon$. More precisely: if $\mathcal{X}$ contains $m$ points, calculating the full persistence diagram for $VR$ in the worst case has time complexity $O({{m \choose p+2}^\omega})$ where $\omega = 2.376$ is the best known exponent for matrix multiplication \cite{phmmtime}.
	
	The Persistence Algorithm has been significantly optimized since its original formulation (for instance: \cite{BauerChunks,TannerParallel,ChenTwist}). Despite improvement in optimizations and implementations, limiting both the size of the point cloud $m$ and the homology dimension $p$ is often necessary in practice to make persistent homology computations feasible (see \cite{roadmap}). Many applications restrict to computing PH only in dimensions $p\leq 2$. 
	Since memory consumption grows rapidly as the number of points $m$ increases,
	this necessitates keeping the size of point samples as low as~possible.
	
	\begin{figure}[h]
		\centering
		\begin{tabular}{lcr}
			\includegraphics[scale=0.33]{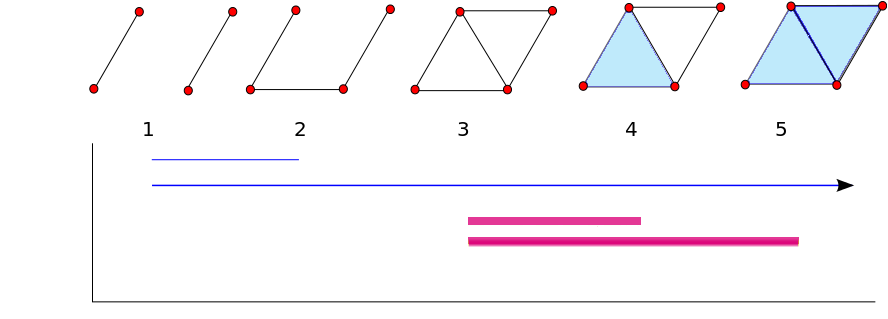} & &
			\includegraphics[scale=0.38]{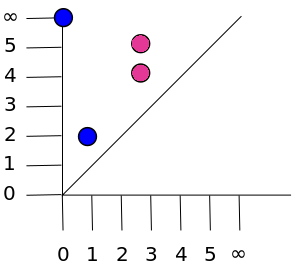}
		\end{tabular}
		\caption{\label{barfig}The persistence diagram and barcode of a filtered complex. The left figure shows the complex as it changes with parameter value, along with the corresponding functor's barcode. The right figure depicts the persistence diagram equivalent to the barcode. Blue bars and points represent $0$ dimensional homology, whereas pink bars and points represent $1$ dimensional homology. An arrow on a bar indicates that the homology feature corresponding to the bar ``lives forever"- the corresponding interval is of the form $[a,\infty)$.}
	\end{figure}

	\subsection{Homology inference}
	
	Suppose that $\mathcal{X}\subseteq\mr^n$ is a finite point cloud sampled from nearby the compact topological space $X\subseteq\mr^n$. A key property of persistent homology, first observed and proven in \cite{harertop}, is that persistent homology computed from $\mathcal{X}$ recovers the homology of the $X$ provided $X$ is a ``dense enough'' sample. To make this notion precise, recall that any compact topological space such as $X$ defines the distance-to-$X$ function $d_X:\mr^n\to\mr$. The function is given by $d_X(y) = \min_{x\in X} d(x,y)$ for any $y\in\mr^n$. Given any real number $\epsilon\geq 0$, define $X^\epsilon = d_X^{-1}(-\infty,\epsilon]$. The space $X^\epsilon$ is formed from $X$ by taking the union of all closed balls of radius $\epsilon$ in $\mr^n$ centered at points of $X$.

	\begin{definition}
		Let $A,B \subseteq \mr^n$ be compact and $0\leq \delta\leq\epsilon \in\mr$. 
		The set $A$ is a \emph{$(\delta,\epsilon)$-sample} of $B$ if $A \subseteq B^\delta$ and $B\subseteq A^\epsilon$.\label{densitydef}
	\end{definition}

	\begin{rem} \Cref{densitydef} is a specific instance of an \emph{interleaving} between generalized persistence modules as defined in \cite{bubenik2015metrics}. It is also generalization of the Hausdorff distance between subsets of metric space. Given compact $A,B\subseteq \mr^n$, the smallest $\epsilon$ such that $A$ and $B$ are $(\epsilon,\epsilon)$-samples of one another is the Hausdorff distance between $A$ and $B$.
	\end{rem}
	
	\begin{definition} Let $X \subseteq \mr^n$ be a compact metric space. The homological feature size of $X$, $\hfs(X)$, is the smallest critical value over all dimensions $k$ of the persistence module $\epsilon \mapsto H_k(X^\epsilon)$ (negative $\epsilon$ are assigned $\emptyset$). Equivalently $\hfs(X)$ is the smallest number $\epsilon$ such that thickening the space $X$ by $\epsilon$ changes its homology.
	\end{definition} 
	
	\begin{rem} Homological feature size of a space was introduced in \cite{harertop}, and is bounded below by the space's \emph{local feature size} \cite{amenta1999surface} and \emph{weak feature size} \cite{chazal2009sampling}. It follows from the definition that $\hfs(X) \geq 0$ for any space. The weak feature size of real semialgebraic sets is known to be positive (\cite{FuTubular} \S 5.3), and so the homological feature size of real algebraic varieties is positive as well.
	\end{rem} 
	
	\begin{theorem}[Homology Inference Theorem, \cite{harertop,chazalwfs}]\label{homthm} Let $\mathcal{X},X\subseteq \mr^n$, with $X$ compact and $\mathcal{X}$ a finite $(\delta,\epsilon)$-sample of $X$ where $0\leq\delta\leq\epsilon$ and $\hfs(X) > 2(\epsilon+\delta)$. Letting $HC = H_pC_\bullet(\mathcal{X})$, the dimension of $H_p(X)$ is the number of points in $D(HC)$, the persistence diagram of the functor $HC$, above and to the left of the point~$(\epsilon,2\epsilon+\delta)\in\bar\mr^2$.
	\end{theorem} 
	\begin{proof}
		From the definition of $(\delta,\epsilon)$-sample we have inclusions $X\hookrightarrow \mathcal{X}^\epsilon \hookrightarrow X^{\epsilon+\delta} \hookrightarrow \mathcal{X}^{2\epsilon + \delta} \hookrightarrow X^{2(\epsilon+\delta)}$. The Nerve Theorem implies that $HC(a) \cong H_p(\mathcal{X}^a)$ for all $a\in\mr$. Applying homology to the sequence and using the assumption on the homology when thickening $X$, we obtain the commutative diagram 
		
		\[
		\begin{tikzcd}  
		H_p(X) \arrow{r} &
		HC(\epsilon) \arrow[bend right=20,swap]{rr}{h} \arrow{r} &
		H_p(X) \arrow{r} &
		HC(2\epsilon + \delta) \arrow{r} & 
		H_p(X)
		\end{tikzcd}
		\]
		
		where the maps from $H_p(X)$ to itself are isomorphisms. Consider the map $h$. Since there is an isomorphism from $H_p(X)$ to itself which factors through $h$, $\dim H_p(X) \leq \rank(h)$. We also have that $h$ factors through a map with domain $H_p(X)$, so that $\rank(h) \leq \dim H_p(X)$. Thus $\rank(h) = \dim H_p(X)$. The Theorem follows upon noting that $\rank(h)$ counts the number of intervals in the barcode for $HC$ with left endpoint at most $\epsilon$, and right end point greater than~$2\epsilon+\delta$. These intervals correspond to points above and to the left of $(\epsilon,2\epsilon+\delta)$ in the persistence~diagram. 
	\end{proof}
	
	\begin{corollary}\label{inferencecor} 
		Let $HC,X,\mathcal{X},\epsilon$, and $\delta$ be as in \cref{homthm}. Then the number of points above and to the left of $\left(2\epsilon\sqrt{\frac{n+1}{2n}},4\epsilon+2\delta\right)$ 
		in the persistence diagram for \hbox{$VR = H_pR_\bullet(\mathcal{X})$}
		is a lower bound for $\dim H_p(X)$ and upper bound for $\rank\left( HC\left( \epsilon\sqrt\frac{n+1}{2n} \leq (2\epsilon+\delta)\sqrt{\frac{2n}{n+1}}\right)\right)$. 
	\end{corollary} 
	\begin{proof}
		Let $a = 2\epsilon\sqrt{\frac{n+1}{2n}}$. By \cref{interthm}, we have the following commutative diagram of linear~maps
		\[
		\begin{tikzcd} 
		VR(a) \arrow{r} \arrow[bend right=20,swap]{rrr}{h} & 
		HC(\epsilon) \arrow{r} & 
		HC(2\epsilon+\delta) \arrow{r} & 
		VR(4\epsilon+2\delta)\text{.} 
		\end{tikzcd} 
		\] 
		It follows that $\rank(h) \leq \rank(HC(\epsilon \leq 
		2\epsilon+\delta))$ because $h$ factors through $HC(\epsilon \leq 2\epsilon+\delta)$. \Cref{homthm} shows that the rank of $HC(\epsilon \leq 2\epsilon+\delta)$ is $\dim H_p(X)$. As in the proof of \cref{homthm}, the rank of $h$ is precisely the number of points in the persistence diagram of $VR$ which are above and to the left of $(a,4\epsilon+2\delta)$. The proof for the upper bound is similar, and uses the sequence of maps $HC(\frac{a}{2}) \to VR(a) \to VR(4\epsilon + 2\delta) \to HC((2\epsilon+\delta)\sqrt{\frac{2n}{n+1}})$.
	\end{proof}
	
	\begin{figure}[!hb] 
		\centering
		\setlength{\aboverulesep}{0pt}
		\setlength{\belowrulesep}{0pt}
		\begin{tabular}{|c|c|}
			\hline
			\includegraphics[width=0.45\textwidth,height=0.45\textwidth,keepaspectratio]{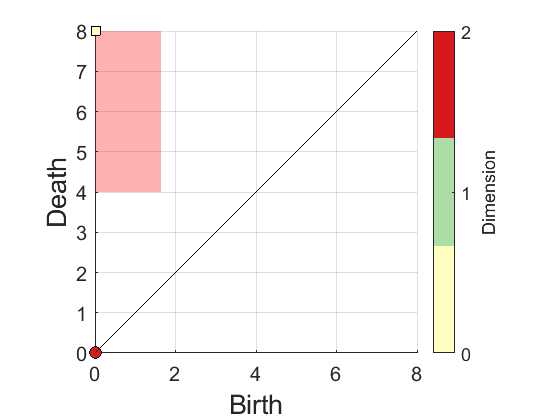} & \includegraphics[width=0.45\textwidth,height=0.45\textwidth,keepaspectratio]{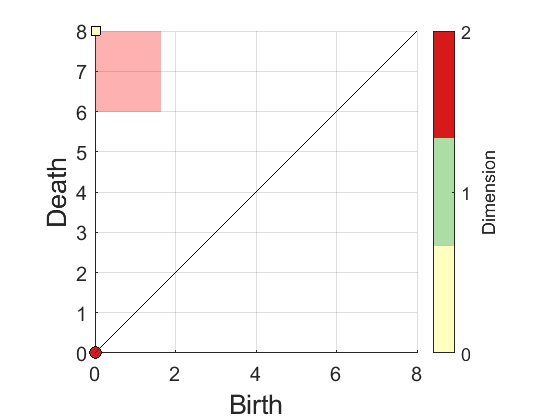} \\
			\hline 
		\end{tabular}
		
		\caption{Diagrams displaying the results of \cref{inferencecor}. The diagram on the left is for a $(0,1)$-sample of an underlying space with homological feature size at least 2. Any points falling into the pink region in the diagram on the left correspond to features in the underlying space. The diagram on the right is the same, but is for a $(1,1)$-sample of an underlying space with homological feature size at least 4. \label{fig:inference}}
	\end{figure} 
	
	
	\section{Sampling using numerical algebraic geometry}\label{Sec:Sampling}
	
	An algebraic variety $V\subset\bC^N$ is the solution
	set of a system of polynomial equations.  
	The real points of $V$, 
	\mbox{$V_\bR = V\cap\bR^N\subset\bR^N$}, 
	is a real algebraic variety.
	One approach to compute a point on $V_\bR$
	is by computing a point $x\in V_\bR$ which is a global minimizer of the 
	distance function between a given
	test point $y\in\bR^N$ and $V_\bR$ \cite{Seidenberg}.
	We summarize the use of numerical algebraic geometry to perform this computation based on \cite{RealMinDistance} (see also \cite{Mohab1,Mohab2,EuclideanDistance})
	with \cref{Sec:Generate} relying on this to
	generate a provably dense sampling of $V_\bR$.
	
	Suppose that $f_1,\dots,f_{N-d}\in\bR[x_1,\dots,x_N]$ 
	and let $V\subset\bC^N$ be the union of $d$-dimensional 
	irreducible components of the solution set of 
	$f = \{f_1,\dots,f_{N-d}\} = 0$.
	That is, $V$ is a pure \mbox{$d$-dimensional} algebraic variety
	with corresponding real algebraic variety $V_\bR = V\cap\bR^N$.
	We note that there is no loss of generality since one 
	can utilize randomization if more than $N-d$ polynomials are provided
	as shown in the following example.
	
	\begin{example}
		The affine cone over the twisted cubic curve
		is the irreducible surface ($d = 2$)
		$$V = \{(s^3,s^2t,st^2,t^3)~|~s,t\in\bC\}\subset\bC^4$$
		which is defined by $g_1 = g_2 = g_3 = 0$ where
		$$g(x) = \left[\begin{array}{c} x_2^2 - x_3 x_1 \\ x_2x_3 - x_4x_1 \\ x_2x_4 - x_3^2 \end{array}\right].$$
		Since $N = 4$, we can randomize down to $N-d = 2$ equations, say
		$f_1 = f_2 = 0$~where
		$$f(x) = \left[\begin{array}{c} x_2^2 - x_3x_1 + 2(x_2x_4 - x_3^2) \\ x_2x_3 - x_4x_1 - 3(x_2x_4 - x_3^2) 
		\end{array}\right].$$
		In particular, $V$ is one of the two irreducible components of 
		the solution set defined by $f_1 = f_2 = 0$
		with the other being the plane 
		defined by $3x_1 + 7x_2 - 4x_4 = x_1 - 7x_3 - 6x_4 = 0$.\done
	\end{example}
	
	\setcounter{theorem}{2}
	
	Given a test point $y\in\bR^N$, the approach of Seidenberg \cite{Seidenberg}
	is to compute a global minimizer~of
	\begin{equation}\label{eq:Minimize}
	\min\left.\left\{\sum_{i=1}^N (x_i-y_i)^2~\right|~x\in V_\bR\right\}
	\end{equation}
	which is accomplished by solving the Fritz John optimality conditions, namely solving
	$$G_y(x,\lambda) = \left[\begin{array}{c} f(x) \\ 
	\lambda_0 (x - y) + \sum_{i=1}^{N-d} \lambda_i \nabla f_i(x)
	\end{array}\right]$$
	on $\bC^N\times \bP^{N-d}$, where $\nabla f_i(x)$ is the gradient of $f_i(x)$ with respect to $x$.  The polynomial system~$G_y$ is a so-called square system
	consisting of $(N-d)+N$ polynomials on $\bC^N\times \bP^{N-d}$ so
	that it is amenable to solving via homotopy continuation.  
	In particular, for $\beta\in\bC^{N-d}$, consider
	$$H_{y,\beta}(x,\lambda,t) = \left[\begin{array}{c} f(x) - t\beta \\ 
	\lambda_0 (x - y) + \sum_{i=1}^{N-d} \lambda_i \nabla f_i(x)
	\end{array}\right].$$
	The following is immediate from coefficient-parameter continuation \cite{CoeffParam}
	showing that generic choices of parameter values $(y,\beta)$ leads to a well-constructed homotopy $H_{y,\beta}$.
	\begin{proposition}\label{prop:CoeffParam}
		There exists a nonempty Zariski dense open subset $U\subset\bC^{N}\times\bC^{N-d}$ 
		such that if $(y,\beta)\in U$, then 
		\begin{enumerate}
			\item the set $S\subset\bC^N\times\bP^{N-d}$ consisting of all solutions
			to $H_{y,\beta}(x,\lambda,1) = 0$ is finite and each is a nonsingular solution;
			\item the number of points in $S$ is equal to the maximum number,
			as $y'\in\bC^N$ and $\beta'\in\bC^{N-d}$ both vary,
			of isolated solutions of $H_{y',\beta'}(x,\lambda,1) = 0$;
			\item the solution paths defined by the homotopy \\
			$H_{y,\beta}(x,\lambda,t) = 0$ 
			starting at the points in $S$ at $t = 1$ are smooth for $t\in(0,1]$.
		\end{enumerate}
	\end{proposition}
	
	Since $G_y(x,\lambda) = H_{y,\beta}(x,\lambda,0)$, 
	the endpoints of solution paths defined by $H_{y,\beta}(x,\lambda,t) = 0$ 
	contained in $\bC^N\times\bP^{N-d}$ are solutions of $G_y = 0$.
	Hence, by tracking the finitely many paths starting at the points
	in $S$ at $t = 1$, one obtains a finite set of solutions of $G_y = 0$,
	one of which corresponds with the global minimizer of \eqref{eq:Minimize}
	as shown in the following from \cite[Thm.~5]{RealMinDistance}.
	
	\begin{theorem}\label{parhomtheorem}
		Suppose that $y\in\bR^N$ and $\beta\in\bC^{N-d}$ such that
		the three items in \cref{prop:CoeffParam} hold.
		Let~$E$ be the set of endpoints contained in $\bC^N\times\bP^{N-d}$ 
		of the homotopy paths starting at the points of~$S$ at $t = 1$
		defined by $H_{y,\beta}(x,\lambda,t) = 0$
		and $\pi_1(x,\lambda) = x$.
		Then, $\pi_1(E)\cap V_\bR$ contains finitely many points, one
		of which is a global minimizer of \eqref{eq:Minimize}.
		Hence, $V_\bR = \emptyset$ if and only if $\pi_1(E)\cap V_\bR = \emptyset$.
	\end{theorem}
	
	Since $\pi(E)\cap V_\bR$ consists of finitely many points,
	a global minimizer of \eqref{eq:Minimize} is identified
	by simply minimizing over these finitely many points.


	\section{Generating samples}\label{Sec:Generate}
	
	This section presents an algorithm integrating \cref{parhomtheorem} with geometric tools to produce provably dense samples of real algebraic varieties. The input and output of the algorithm are as follows. \\ \\
	\noindent {\bf Input}: 
	\begin{itemize}
		\item Polynomial equations $f_1,\dots,f_{N-d} \in\mr[x_1,\dots,x_N]$ defining a pure $d$-dimensional real algebraic variety $X = \mvr(f_1,\dots,f_{N-d})$.
		\item A compact region $R\subseteq\mr^N$ of the form $R = [a_1,b_1]\times\dots\times [a_N,b_N]$. We call any regions of this form boxes.
		\item A sampling density $\epsilon > 0$.
		\item An estimation error $\delta$ with $0 \leq \delta \leq \epsilon$.
	\end{itemize}
	
	\noindent {\bf Output}: A (finite) set of points $\mathcal{X}\subseteq \mr^N$ that form a $(\delta,\epsilon)$-sample of $X\cap R$.
	
	\smallskip
	
	\Cref{prop:CoeffParam} and \cref{parhomtheorem} provide a computationally tractable approach to finding very accurate estimated solutions of the optimization problem \eqref{eq:Minimize} for generic $y\in\mr^N$. Following the terminology of these two results, we can define a subroutine $\texttt{MinDistance}$ which takes a point $y\in\mr^N$ as input, and outputs a set $S$ consisting of one point $s_q$ with $d(q,s_q)\leq \delta$ for every point $q\in\pi_1(E)\cap V$. The subroutine follows these steps on input $y$: 
	\begin{enumerate} 
		\item Choose a parameter $\beta\in\mc^{N-d}$ such that \cref{parhomtheorem} holds for the pair $(y,\beta)$, which
		exists for generic $y$, and construct the homotopy $H_{y,\beta}$ using the polynomial system $f$ defining $X$.
		\item Track the homotopy paths of $H_{y,\beta}$, which are guaranteed to exist by \cref{parhomtheorem}, to obtain the elements of $\pi_1(E)\cap V_\mr$ up to numerical error $\delta$.
	\end{enumerate} 
	
	The smallest distance from $y$ to a point in $\texttt{MinDistance}(y)$ solves the problem \eqref{eq:Minimize} up to error $\delta$. Repeatedly solving the minimum distance problem this way yields enough information to construct a provably dense sampling of $X$. Neglecting estimation error momentarily, the sampling algorithm's core consists of a short loop which computes the desired sample points iteratively. Denoting the open ball of radius $r$ centered at $y$ by $B_r(y)$ for any $r > 0$, this short loop is:
	
	\begin{enumerate}
		\item Choose an appropriate new ``test point" $y\in\mr^N$.
		\item Run $\texttt{MinDistance}(y)$ and place the returned points into the set of output points. Each sample point $s$ covers a region $B_\epsilon(s)$ of points in $X$ that are within distance $\epsilon$ of $s$. Let $d=d_X(y)$ be the minimum distance from $y$ to $X$ which can be calculated from the points returned by $\texttt{MinDistance}(y)$. 
		Thus, the region $B_d(y)$ does not contain any points of $X$. Store information about $B_d(y)$ and $B_\epsilon(s)$ (for all returned sample points $s$) for later use.
		\item Check to see if the union all of the regions of the form $B_\epsilon(s)$ and $B_d(y)$ found in previous iterations of Step 2 contains $R$. If so, stop and output the sample points which have been collected. Otherwise, return to Step 1. 
	\end{enumerate}
	
	\begin{rem}\label{algremark} The stopping condition in Step 3 above guarantees that the outputted sample is a dense sample of $X\cap R$. Suppose that $R\subseteq \mathcal{B}\cup\mathcal{C}$, where $\mathcal{B} = \cup_{s\in S} B_\epsilon(s)$ for some subset $S$ of $X$, and $\mathcal{C}\cap X = \emptyset$. Then for any $x\in X\cap R$, it follows that $x\in\mathcal{B}$, so $d(x,s_0) < \epsilon$ for some $s_0\in S$. Thus, $d_S(x) < \epsilon$. 
	\end{rem}
	
	The full sampling algorithm tracks the spatial information for Steps 1 and 3 above by recursively dividing the region $R$ into smaller boxes as necessary. Let $\texttt{SplitBox}$ be a subroutine which takes as input a box $A = [c_1,d_1]\times\dots\times [c_N,d_N] \subseteq \mr^N$. It returns a pairwise disjoint set of smaller boxes $\{A_1,\dots,A_k\}$ such that $A = \cup_{i=1}^k A_i$. We can arrange repeated applications of $\texttt{SplitBox}$ into a tree~structure. 
	
	\begin{definition} Let $T_R$ be a tree with root $R$ whose nodes are boxes in $\mr^N$. The children of $R$ in $T_R$ are the elements of $\texttt{SplitBox}(R)$. Suppose that all the $(n-1)$-children of $R$ in $T_R$ have been defined where $n > 1$. Then the $n$-children of $R$ are the elements of $\texttt{SplitBox(C)}$ for every $(n-1)$-child $C$ of $R$. The elements of $\texttt{SplitBox(C)}$ have parent node $C$.
	\end{definition}
	
	For technical reasons, repeated applications of $\texttt{SplitBox}$ must eventually split an input region $A = [c_1,d_1]\times\dots\times[c_N,d_N]$ into arbitrarily small pieces. Put precisely, given any $\gamma > 0$ and input region $A$, there is some $n$ such that all $n$-children of $A$ in $T_A$ have maximum side length at most~$\gamma$. As an example, consider a version of $\texttt{SplitBox}$ that when applied to $A$ returns the two boxes $[c_1,d_1]\times\dots\times[c_j,\frac{c_j+d_j}{2}] \times \dots \times [c_N,d_N]$ and $[c_1,d_1]\times\dots\times[\frac{c_j+d_j}{2},d_j]\times\dots\times [c_N,d_N]$ where $\vert d_j - c_j \vert$ is the maximum side length for the box $A$. Using $\texttt{SplitBox}$ the sampling algorithm conducts a breadth first search of $T_R$ while iteratively building the output sample.

	\setcounter{theorem}{3}
	\begin{algorithm}
		\begin{algorithmic}[1]
			\State Initialize an empty spatial database $\textsc{CoveredRegions}$ which can store and retrieve information about subregions of $\mr^N$
			\State Initialize an empty list $\textsc{SampleOutput}$ of points in $\mr^N$
			\For{each node $M$ in $T_R$ not marked ``done", iterated via breadth first search}
			\If{The maximum side length of $M$ is at most $\frac{\epsilon-\delta}{\sqrt{N}}$ or $M$ does not intersect any region stored in $\textsc{CoveredRegions}$} 
			\State Run $\texttt{MinDistance(y)}$ where $y$ is the center point of $M$, returning a set of sample points
			\State $S$ with minimum distance $d$ from $y$ to any point in $S$. Add regions $B_{d-\delta}(y)$ and
			\State $B_{\epsilon}(s)$ for each $s\in S$ to $\textsc{CoveredRegions}$. Add each $s\in S$ to $\textsc{SampleOutput}$.
			\EndIf
			\If{$M\subseteq B$ for any region $B$ contained in $\textsc{CoveredRegions}$} 
			\State Mark $M$ and all nodes in the subtree rooted at $M$ ``done" and stop  \State searching the subtree rooted at $M$.
			\EndIf
			\If{All unsearched boxes in $T_R$ are marked ``done"} 
			\State End loop.
			\EndIf
			\EndFor
			\State \Return $\textsc{SampleOutput}$
		\end{algorithmic}
		\caption{\textsc{Sampling algorithm} }
		\label{sampalgorithm}
	\end{algorithm}
	
	\setcounter{theorem}{3}
	
	\begin{theorem}\Cref{sampalgorithm} terminates and outputs a $(\delta,\epsilon)$-sample of $X\cap R$. \label{algtheorem}
	\end{theorem}
	
	Before proving \cref{algtheorem}, we consider the following.
	
	\begin{lemma} 
		(1) Let $A$ be a box in $\mr^N$ with maximum side length less than $\frac{\epsilon-\delta}{\sqrt{N}}$. If $y$, $S$, and~$d$ take values as in lines 5-7 of \cref{sampalgorithm}, then either $A \subseteq B_{\epsilon}(s)$ where $s\in S$ minimizes the distance to $y$, or $A \subseteq B_{d-\delta}(y)$.
		(2) Let $T_A$ be a tree of boxes in $\mr^N$ with root $A$ constructed via $\texttt{SplitBox}$ in the same manner as~$T_R$, and let $T_A'$ be a finite subtree of $T_A$ such that if a node $M$ is in $T_A'$ and is not a leaf, all the first children of $M$ in $T_A$ are contained in $T_A'$. If $\mathcal{L} = \{L_1,\dots,L_k\}$ are the leaf nodes of $T_A'$, the equality $A = \cup_{i=1}^k C_i$ follows. 
		\label{alglemma} 
	\end{lemma} 
	\begin{proof} 
		(1): Let $\gamma$ be the maximum side length of $A$ and suppose that $y=(y_1,\dots,y_N)^T$. Without loss of generality we can replace $A$ with the hypercube $\Pi_{i=1}^N [y_i-\frac{\gamma}{2},y_i+\frac{\gamma}{2}]$ since $A$ is a subset of the latter box. $A$ has diagonal length $\Delta = \gamma\sqrt{N}$, which by assumption is less than $\epsilon - \delta$. Let $a\in A$ be an arbitrary point, and note that the maximum distance from $a$ to $y$ is half the length of the diagonal $\Delta$. Suppose that $d = d(y,s) \leq \frac{\Delta}{2} + \delta$. Then for any point in $a\in A$, it follows that $d(a,s) \leq d(y,s) + d(y,a) \leq \Delta + \delta < \epsilon - \delta + \delta = \epsilon$. Therefore $A \subseteq B_{\epsilon}(s)$. Otherwise, suppose $d = d(y,s) > \frac{\Delta}{2}+\delta$. Then $d-\delta > \frac{\Delta}{2}$. Since the maximum value of $d(y,a)$ is $\frac{\Delta}{2}$, it follows that $a\in B_{d-\delta}(y)$, so $A \subseteq B_{d-\delta}(y)$. 
		\\ \\ 
		(2): We proceed by induction on the maximum depth of the tree $T_A'$. Suppose that the depth of $T_A'$ is $0$. Then $T_A'$ is a tree that consists of one leaf node, the box $A$, and (2) holds trivially. Suppose that (2) holds for any box $B$, corresponding tree $T_B$, and subtree $T'_B$ with depth at most $k-1$ where $k\geq 1$. Then if $T_A'$ has depth $k$, $T_A'$ contains all the nodes $\texttt{SplitBox}(A) = \{A_1,\dots,A_j\}$ by assumption. Note that $T_A'$ is the union of the root $A$ along with finite subtrees fulfilling the conditions of (2) rooted at $A_1,\dots,A_j$, and that the set $\mathcal{L}$ of leaf nodes of $T_A'$ is the union $\mathcal{L}_1\cup \dots \cup \mathcal{L}_j$ where $\mathcal{L}_i$ is the set of leaf nodes of the subtree rooted at $A_i$. By the induction assumption, $\cup_{L\in \mathcal{L}_i} L = A_i$. Therefore $A = \cup_{i=1}^j A_i = \cup_{i=1}^j \cup_{L\in\mathcal{L}_i} L = \cup_{L\in\mathcal{L}} L$ as desired. 
	\end{proof}
	
	\begin{proof}[ Proof of \cref{algtheorem}]
		(Termination): Let $\alpha = \frac{\epsilon-\delta}{\sqrt{N}}$. By our assumption on $\texttt{SplitBox}$ there is an~$n$ such that the $n$-children of $R$ in $T_R$ have side length less than $\alpha$. Therefore if $M$ is any $n$-child of $R$, lines 5-7 of the algorithm will run if $M$ is searched. Part (1) of \cref{alglemma} shows that lines~10-11 will run on $M$ subsequently. Therefore the algorithm's breadth first search terminates at maximum depth $n$. 
		\\ \\
		\noindent(Completeness): Let $T'_R$ be the subtree of $T_R$ which \cref{sampalgorithm} searches 
		before terminating. By construction, $T'_R$ fulfills the conditions of \cref{alglemma} 
		part (2). If $\sL$ is the set of leaf nodes in $T'_R$, then $R = \cup_{L\in \sL} L$ 
		follows. Let $S$ be $\textsc{SampleOutput}$ which was returned by the algorithm and $Y$ be 
		the set of center points of balls with form $B_{d-\delta}(y)$ in $\textsc{CoveredRegions}$. 
		By construction any element $L\in\sL$ has $L\subseteq B_{\epsilon}(s)$ for some $s\in S$ or 
		$L\subseteq B_{d-\delta}(y)$ for some $y\in Y$. By \cref{parhomtheorem} and the definition of $\texttt{MinDistance}$ it follows 
		that $X\cap (\cup_{y\in Y} B_{d-\delta}(y)) = \emptyset$. Similarly to \cref{algremark}, 
		we have that $X\cap R \subseteq \cup_{s\in S} B_{\epsilon}(s)$. We also have $d_X(s) \leq \delta$ for all $s\in S$ by definition of $\texttt{MinDistance}$. Thus $S$ is a $(\delta,\epsilon)$-sample of $X\cap R$. 
		
	\end{proof}

	In practice, there are two opposing resource demands the algorithm needs to balance. The $\texttt{MinDistance}$ step in \cref{sampalgorithm}'s core loop consumes significantly more time than any other individual step, so an optimal run of the Algorithm makes as few calls to $\texttt{MinDistance}$ as possible. Resource demands for processing the Algorithm's output with data analysis methods scale with the number of points in the sample. Also, with more points in the sample more resources are consumed accessing and storing information in the spatial database used throughout the Algorithm. An optimal output sample therefore contains as few points as possible while being provably dense. We can adjust the Algorithm's components, integrating geometric heuristics both to reduce $\texttt{MinDistance}$ calls and output sizes. These heuristics include: 
	\begin{itemize}
		\item \emph{Dynamic box splitting} - Instead of splitting along the longest side of a box $B$ with $\texttt{SplitBox}$, split $B$ so that the largest intersection (by Lebesgue measure) of $B$ with a region stored in $\textsc{CoveredRegions}$ is a box in the output. 
		\item \emph{Dynamic sampling} - Refuse to add points to the output sample if their distance to the nearest point already in $\textsc{SampleOutput}$ is less than some threshold. 
		\item \emph{Heuristic tree searching} - Place priority on first searching and applying $\texttt{MinDistance}$ to the largest boxes (by Lebesgue measure) at each level of depth in the search tree. Larger boxes represent larger regions which potentially do not intersect $X$, and so a single run of $\texttt{MinDistance}$ has the potential to lead to the exclusion of a much larger box $B_{d-\delta}(y)$. 
	\end{itemize}
	See \cite{parkersthesis} for an extended discussion of both the heuristics and implementation.
	
	
	\section{Examples}\label{Sec:Examples}

	\Cref{sampalgorithm} has been implemented and used to produce dense samples of varieties for further processing via persistent homology. The implementation is publicly available as the Python package \texttt{tdasampling} on PyPi and the package source code is available at \url{https://github.com/P-Edwards/tdasampling}. Data, algorithm parameters, plots, and other scripts for the examples are available at \url{https://github.com/P-Edwards/sampling-varieties-data}. Vietoris-Rips persistent homology calculations were performed using the package Ripser \cite{ripser}. Plots of persistence diagrams were produced using a modified version of a plotting script included with DIPHA \cite{dipha}. 
	
	In the following examples, the persistence diagrams are decorated as in \cref{fig:inference}. Points in the highlighted region of an example's diagram correspond to homological features in the underlying variety, assuming the diagram was produced from a $(\delta,\epsilon)$-sample of a variety with homological feature size at least $2(\epsilon+\delta)$.
	
	\subsection{Clifford torus}\label{Sec:CliffordTorus}
	The Clifford torus $T$ is an embedding of the product of two circles, $S^1\times S^1$, into $\mr^4$. It is also a pure 2-dimensional algebraic variety defined by two equations in four variables: 
	$$T = \mvr(x^2_1 + y_1^2 - \frac{1}{2},x_2^2 + y_2^2 - \frac{1}{2}).$$ 
	Since $T$ is a torus, its Betti numbers are known theoretically to be $\beta_0 = 1, \beta_1 = 2,$ and $\beta_2 = 1$. Note that $T$ is compact as it is contained in the closed ball $\widebar{B_1(0)}$ in $\mr^4$. A sample of $T$ was obtained by using \cref{sampalgorithm} to produce a $(10^{-7},0.14)$ sample of $T$ (the bounding box used was $[-1,1]^4$). The sample contains 5,689 points. 
	
	Vietoris-Rips persistent homology thresholded to a parameter value of $0.60$ was subsequently calculated for the sample. The points in the persistence diagram represent features born before $0.60$, and the points on the top edge of the diagram represent features that do not die at $0.60$ or earlier. The shaded region in the diagram is derived from \cref{inferencecor} assuming the homological feature size of the torus is at least $2(0.14+10^{-7})$. Recall from the Corollary that the number of points above and to the left of the point $\left(2\epsilon\sqrt{\frac{4+1}{(2)(4)}},4\epsilon+2\delta\right)$, where the sample is a $(\delta,\epsilon)$ sample of~$T$, is a lower bound on $T$'s Betti numbers. In this case, $\left(2\epsilon\sqrt{\frac{4+1}{(2)(4)}},4\epsilon+2\delta\right) \approx (0.221, 0.56)$. In \cref{torus_diagram}, the shaded region consists of all points above and to the left of $(0.221,0.56)$, the region's bottom right corner. The persistence diagram in \cref{torus_diagram} mirrors the expected theoretical results. A connected component and two 1-dimensional homology features appear in the shaded region, and one long-lived 2-dimensional homology feature also appears in the diagram.
	
	\begin{figure}[h]
		\centering
		\begin{tabular}{|c|c|}
			\hline
			Persistence diagram & 
			\begin{minipage}{0.6\textwidth}
				\includegraphics[width=\linewidth,height=\linewidth,keepaspectratio]{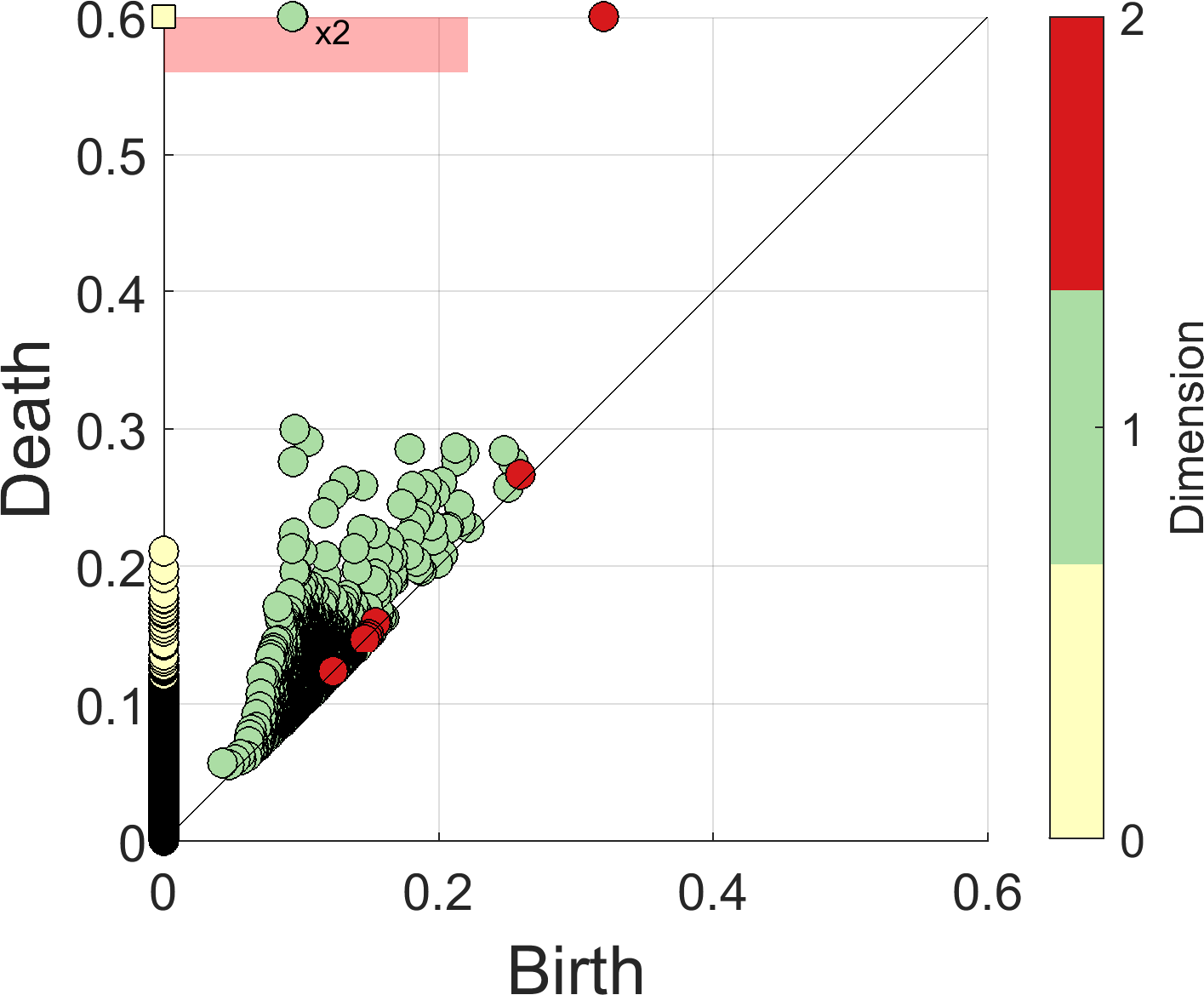} 
			\end{minipage} \\
			\hline 
			Sampling density & $(10^{-7},0.14)$ \\
			\hline 
			Estimated Betti numbers & 
			$\beta_0 = 1,\beta_1 = 2, \beta_2 = 0$ \\
			\hline
		\end{tabular}
		\caption{Persistent homology results dervied from sampling the Clifford torus.  Points in the shaded region of the persistence diagram provably correspond to homology features in the underlying space. \label{torus_diagram}}
	\end{figure}
	
	\subsection{Quartic surfaces}\label{Sec:QuarticSurfaces}
	
	Restricting to the box $[-3,3]\times [-3,3]\times [-3,3]$,
	we next consider the real algebraic varieties 
	\[
	\begin{array}{l}
	V_1 = \mvr(4x^4+7y^4+3z^4-3-8x^3+2x^2y-4x^2-8xy^2-5xy+8x-6y^3+8y^2+4y) 
	\\[0.1in]
	V_2 = \mvr\left(
	\begin{array}{l}
	144x^4+144y^4-225(x^2+y^2)z^2+350x^2y^2\\
	~~~~~~+81z^4+x^3+7x^2y+3x^2+3xy^2-4x-5y^3+5y^2+5y\end{array}\right).
	\end{array}
	\]
	Both quartic equations define pure 2-dimensional varieties.
	\Cref{quartic_sample_fig} displays visualizations of both $V_1$ and $V_2$ using the gathered samples allowing for a qualitative analysis. In particular, $V_1$ appears to be a sphere up to homotopy, with two distinct sphere-like features.

	\begin{figure}[h]
		\centering
		\setlength{\aboverulesep}{0pt}
		\setlength{\belowrulesep}{0pt}
		\begin{tabular}{|c|c|}
			\hline
			$V_1$ & $V_2$ \\
			\includegraphics[width=0.35\textwidth,height=0.35\textwidth,keepaspectratio]{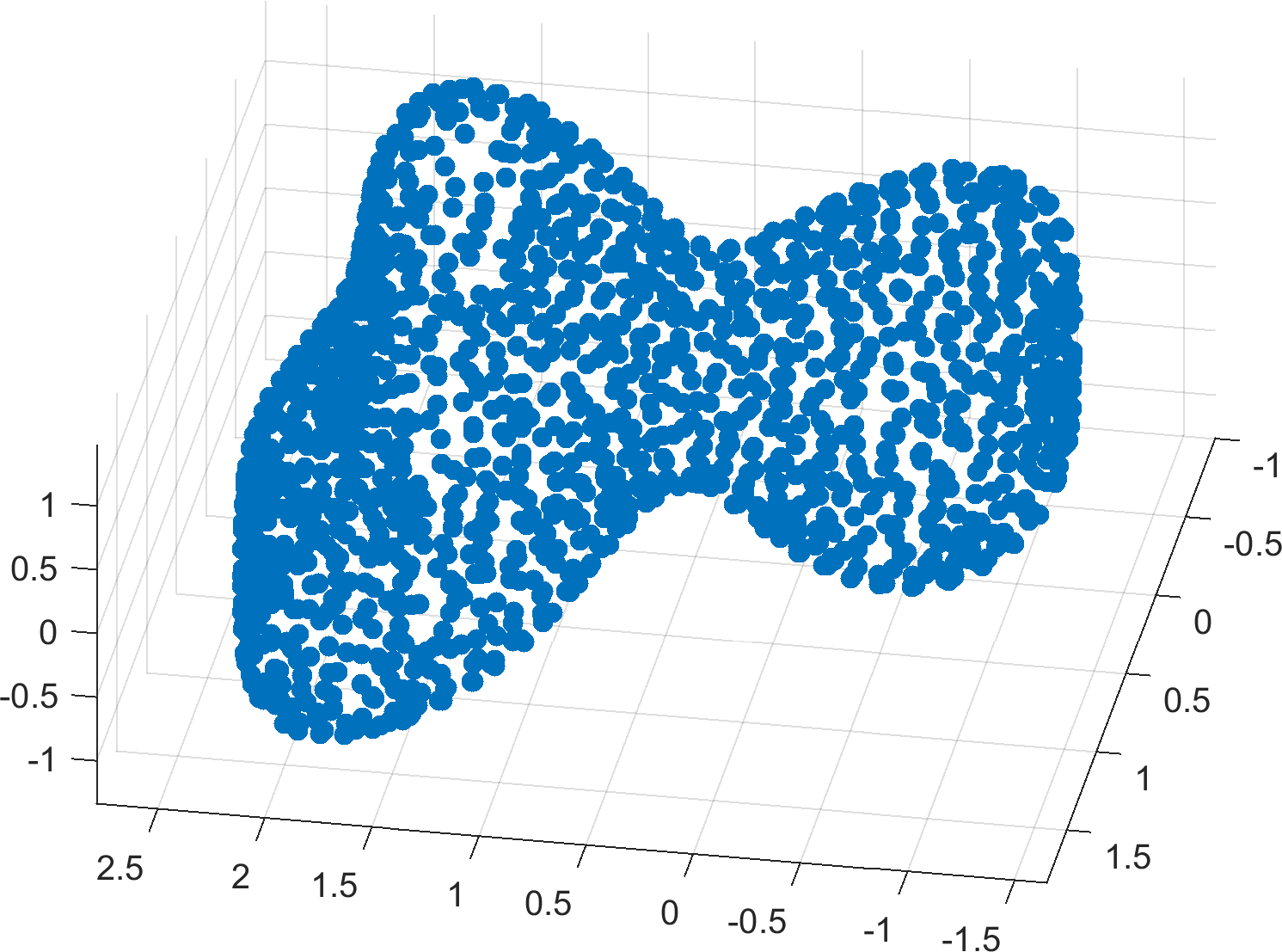} & \includegraphics[width=0.35\textwidth,height=0.35\textwidth,keepaspectratio]{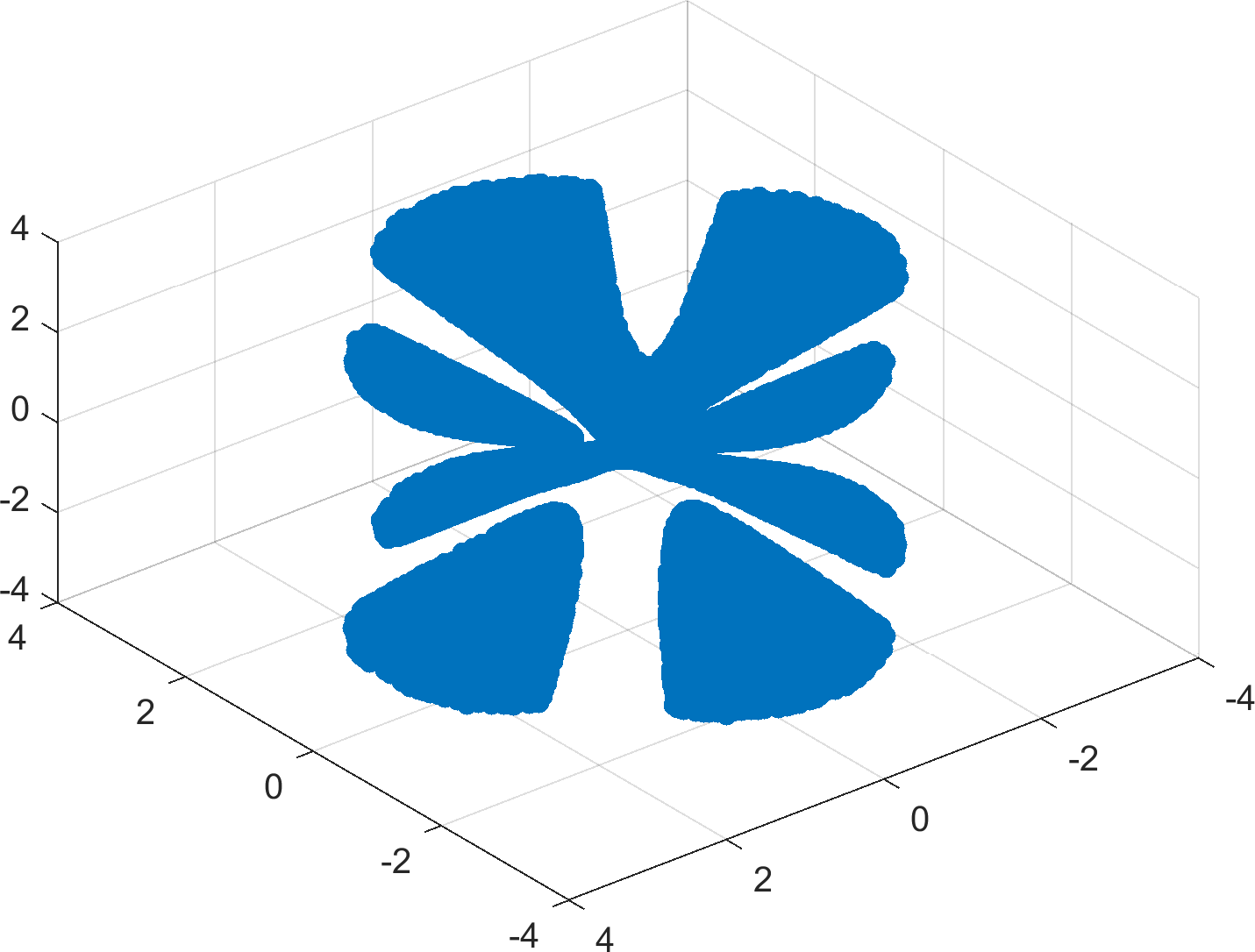} \\
			\hline 
		\end{tabular}
		\caption{Quartic surfaces sampled using \cref{sampalgorithm}. 
			\label{quartic_sample_fig}}
	\end{figure}

	\begin{figure}
		\centering
		\setlength{\aboverulesep}{0pt}
		\setlength{\belowrulesep}{0pt}
		\begin{tabular}{|c|c|}
			\hline
			Variety &
			$V_1$ \\ \hline
			Persistence diagram & 
			\begin{minipage}{0.6\textwidth}
				\includegraphics[width=\linewidth,height=\linewidth,keepaspectratio]{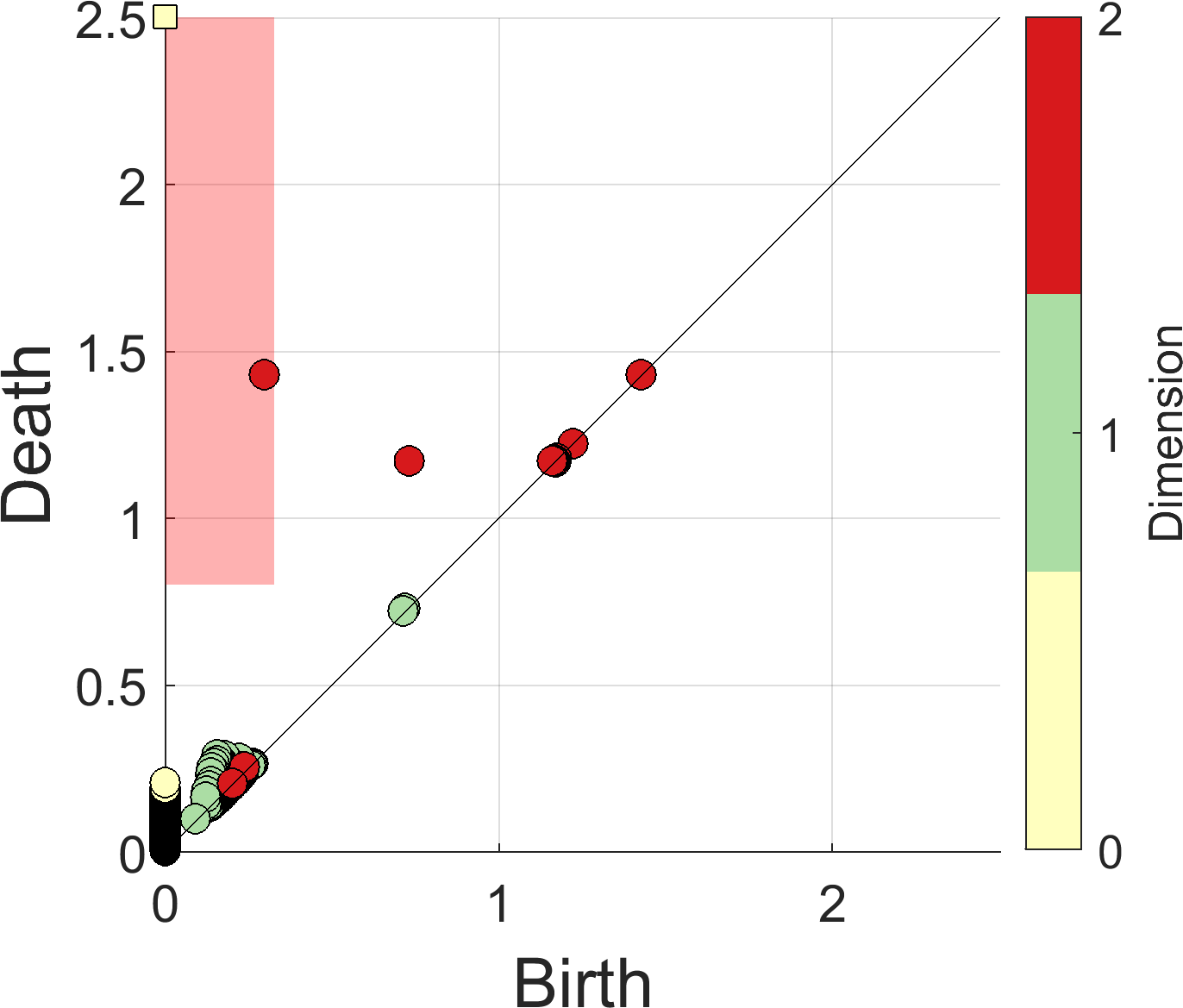} 
			\end{minipage} \\
			\hline 
			Sampling density & $(10^{-7},0.20)$ \\
			\hline
			Estimated Betti numbers & 
			$\beta_0 = 1,\beta_1 = 0, \beta_2 = 1$ 
			\\
			\hline
		\end{tabular}
		\caption{Persistent homology results derived from sampling the variety $V_1$. \label{quartic_diagram_fig1}}
	\end{figure}
	
	\begin{figure}
		\centering
		\setlength{\aboverulesep}{0pt}
		\setlength{\belowrulesep}{0pt}
		\begin{tabular}{|c|c|}
			\hline
			Variety &
			$V_2$ \\ \hline
			Persistence diagram & 
			\begin{minipage}{0.6\textwidth}
				\includegraphics[width=\linewidth,height=\linewidth,keepaspectratio]{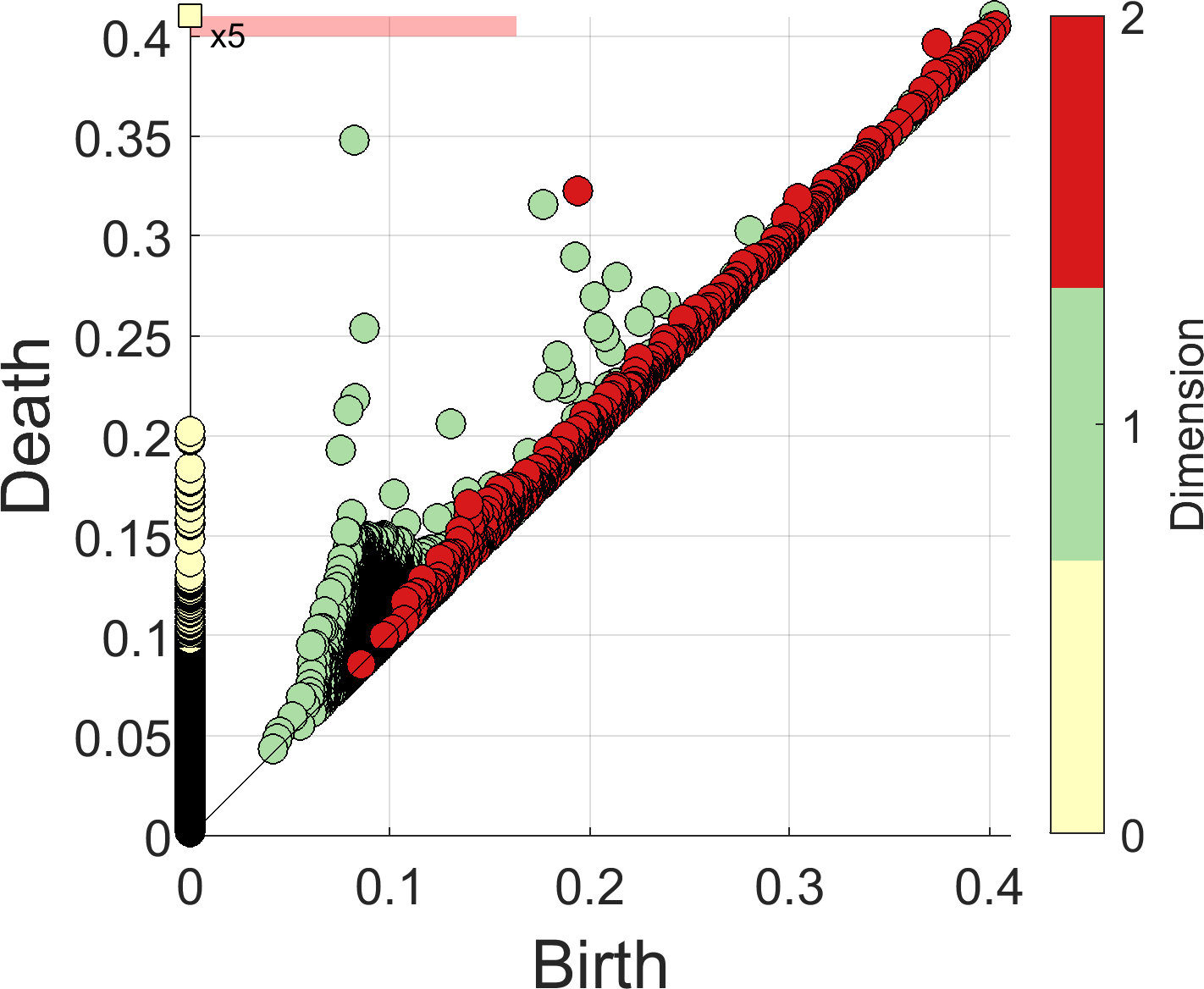} 
			\end{minipage} \\
			\hline 
			Sampling density & $(10^{-7},0.10)$  \\
			\hline
			Estimated Betti numbers & 
			$\beta_0 = 5, \beta_1 = 0, \beta_2 = 0$  
			\\
			\hline
		\end{tabular}
		\caption{Persistent homology results derived from sampling the variety $V_2$. The calculation for $V_2$ has been thresholded to a parameter value of $0.405$. \label{quartic_diagram_fig2}}
	\end{figure}

	Samples produced for $V_1$ and $V_2$ contain 1,511 and 13,904 points respectively. The persistent homology results in \cref{quartic_diagram_fig1} show that $V_1$ has homology features corresponding to a 2-sphere. The persistence diagram for $V_1$ shows how the persistence diagram also captures geometric information about $V_1$ beyond just its Betti numbers. A 2-dimensional point which is relatively far away from the diagonal but not in the shaded region appears in the persistence diagram for $V_1$, and corresponds to the smaller of the two sphere-like features. The only homology features confirmed for $V_2$ in \cref{quartic_diagram_fig2} are 5 connected components.

	\subsection{Deformable pentagonal linkages}\label{Sec:PentLinkage}
	
	For a more elaborate example, we also analyze a kinematics inspired polynomial system. Consider a regular pentagon in the plane consisting of links with unit length, and with one of the links fixed to lie along the $x$-axis with leftmost point at $(0,0)$. The space $V_p$ of all possible configurations of this regular pentagon is a real algebraic variety. Farber and Sch\"{u}tz study this type of configuration space in \cite{farber2007homology}, as well as provide an overview of its study. A specialization of their results shows that $\beta_0$ of $V_p$ is 1, $\beta_1$ is 8, and $\beta_2$ is 1. 
	
	We describe the polynomials defining $V_p$ as explained in \cite{BertiniReal}. Number the links in order around the loop with link 0 as the fixed link, and let $\theta_i$ for $i=0,\dots,4$ be the absolute rotation of the $i$-th link in the plane. Note that $\theta_0 = 0$ because link 0 is fixed, and $\theta_4$ is totally determined by $\theta_i$ for $i=1,2,3$ since one of link~4's endpoints is $(0,0)$ and the other is an endpoint of link 3. 
	A 3-tuple $(\theta_1,\theta_2,\theta_3)$ defines a valid regular pentagon only if the free endpoint of link 3 is distance 1 from $(0,0)$. Letting $s_i = \sin(\theta_i)$ and $c_i = \cos(\theta_i)$ for $i=1,\dots,3$, we have the polynomial condition $(s_1+s_2+s_3)^2 + (1+c_1+c_2+c_3)^2 = 1$. To enforce that $s_i$ and $c_i$ are sine-cosine pairs for $i = 1,\dots,3$, we require the equations $s_i^2 + c_i^2 = 1$. Assembling these constraints together, the configuration space for deformable regular pentagons is modelled by a compact pure 2-dimensional real algebraic variety in the six variables $s_1,s_2,s_3$ and $c_1,c_2,c_3$ with four equations:
	
	\[
	V_p = V_\mr \left(\begin{array}{c}
	s_1^2 + c_1^2  - 1, \\ 
	s_2^2 + c_2^2   - 1, \\ 
	s_3^2 + c_3^2  - 1, \\ 
	(s_1+s_2+s_3)^2 + (1+c_1+c_2+c_3)^2  - 1
	\end{array}\right)\text{.}
	\]
	
	A $(10^{-7},1.12)$ sample of $V_p$ was produced by first obtaining a $(10^{-7},1.0)$ sample using \cref{sampalgorithm}. This sample was then sub-sampled by iteratively choosing a point in the sample, removing all other points within $.12$ of the chosen point, and repeating this loop until all points in the subsample had no other points within distance $.12$. The sample contains 3,548 points, and persistent homology calculations were thresholded to distance value $2.2$. The persistent homology results are summarized in \cref{pentagonal_diagram}. The points far from the diagonal on the left hand side capture the theoretically expected homology for the configuration space. Though 2 features in dimension 2 (voids) appear on the right hand side of the persistent diagram, those features persist for a shorter period of time than the features on the left.

	\begin{figure}[h]
		\centering
		\begin{tabular}{|c|c|}
			\hline
			Persistence diagram & 
			\begin{minipage}{0.6\textwidth}
				\includegraphics[width=\linewidth,height=\linewidth,keepaspectratio]{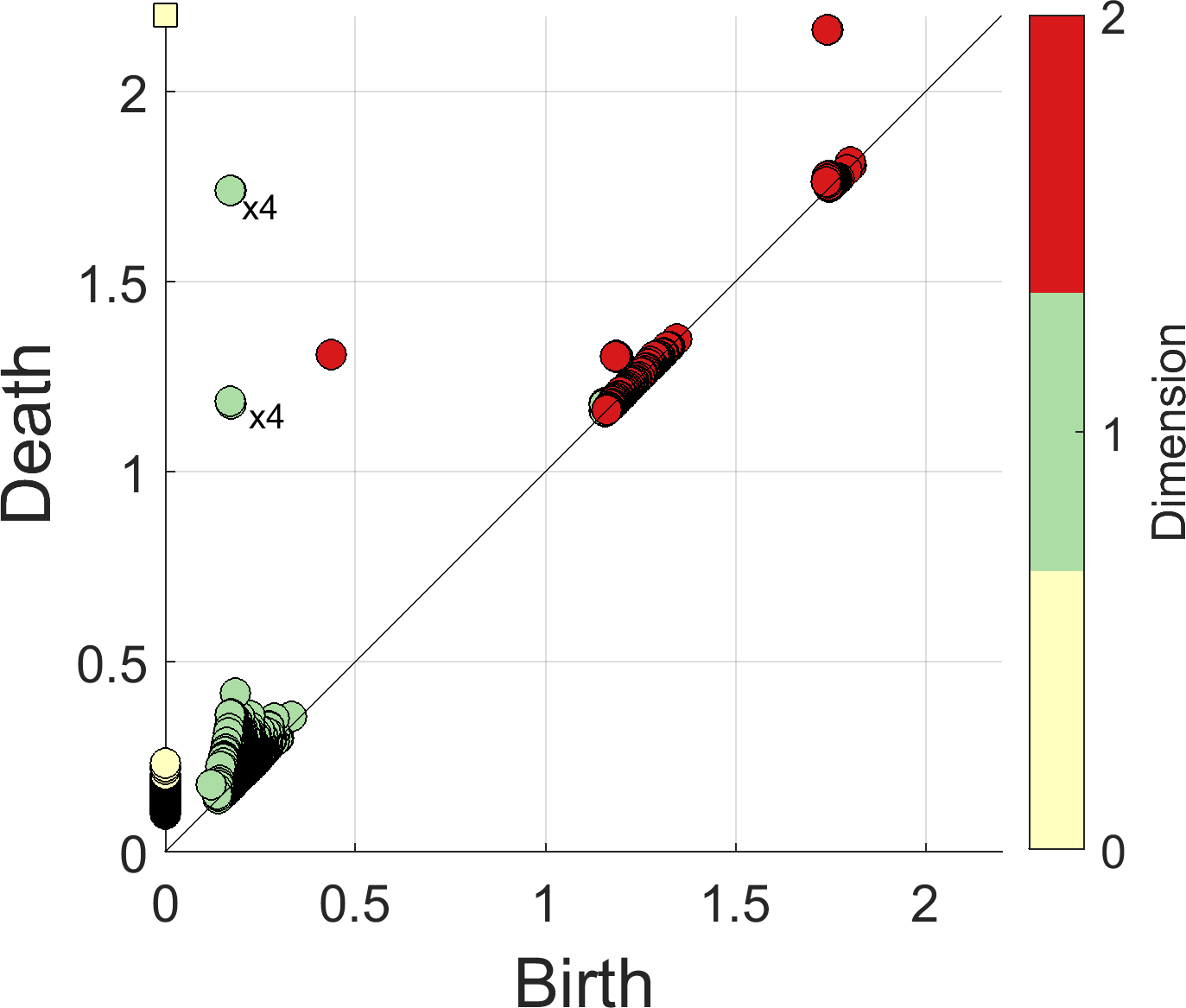} 
			\end{minipage} \\
			\hline 
			Sampling density & $(10^{-7},1.12)$ \\
			\hline 
		\end{tabular}
		\caption{\label{pentagonal_diagram} Persistence diagram computed by sampling the configuration space of deformable pentagonal linkages.}
	\end{figure}


	\section{Conclusion and Future Work}\label{Sec:Conclusion}
	
	The sampling algorithm presented in this paper is a first step towards systematizing the use of the TDA for obtaining geometric and topological information from algebraic varieties, including those that arise in applications. Our use of numerical algebraic geometry methods in the sampling process is unique among sampling approaches, and enables our algorithm to simultaneously satisfy both theoretical and practical constraints for applying TDA. The examples we provide in Section~5 illustrate how using the PH pipeline approach allows for the extraction of detailed information beyond Betti numbers on a real algebraic variety.
	
	A step forward would be to derive and incorporate further information about the geometric structure of singular varieties into systematic TDA based analysis. A real algebraic variety $X$ can be stratified into singular regions: $X\supset X_{0} \supset X_{1}\supset X_{2} \supset \ldots \supset X_t$, where $X_{0}$ is the singular locus of $X$, $X_{1}$ is the singular locus of $X_{0}$, and so on. A classic result of Whitney \cite{WhitneyStructureRealVarieties} shows that this stratification presents $X$ as a stratified manifold. Alternatively, a variety $X$ can also be given an \emph{isosingular stratification} \cite{isosing}, breaking it into strata based on its singularity structure. Running the PH pipeline on individual strata after identifying them via stratification methods for samples (for instance: \cite{WangStratification}) or algebraic methods (detailed in \cite{isosing}) would result in an even more detailed summary of the variety to use for either dimensionality reduction or machine learning. Another direction would be to apply persistent homology of ellipsoids rather than $\epsilon$-balls \cite{breiding2018learning}.
	
	For a variety $X$, our work also raises the natural question of theoretically determining or computationally estimating a lower bound on the weak feature size of $X$. Future work will explore how to exploit the algebraic description of a variety in computing these quantities. Finally, it would be worthwhile to investigate the noise induced from sampling via homotopy continuation in the context of off-set varieties \cite{horobet2018offset}.

	\section*{Acknowledgements}\label{Sec:Ack}
	This paper arises from research done while ED and HAH were funded by the John Fell Oxford University Press (OUP) Research Fund. 
	ED was also funded via a Anne McLaren fellowship from the University of Nottingham.
	PBE thanks Peter Bubenik for helpful template scripts to generate simplicial complex figures and Bernd Sturmfels for the example quartic equations. 
	HAH acknowledges funding from EPSRC EP/K041096/1, EP/R005125/1, EP/R018472/1 and Royal Society University Research Fellowship. 
	JDH was supported in part by NSF CCF 1812746, 
	Army YIP W911NF-15-1-0219, Sloan Research Fellowship BR2014-110 TR14, 
	and ONR N00014-16-1-2722.


	\bibliographystyle{plain}
	\bibliography{TDANAG_references}

\end{document}